\documentclass{article}

\usepackage{graphicx}
\usepackage{textcomp}
\usepackage{amssymb}
\usepackage{amsmath}
\usepackage{amsthm}
\usepackage{enumitem}
\usepackage{caption}
\usepackage{subcaption}
\usepackage{tikz}
\usepackage{nicematrix}
\usepackage{multirow}
\usepackage{mathtools}
\usepackage{anysize}
	\marginsize{2.5cm}{2.5cm}{1.5cm}{2cm}

\usepackage[utf8]{inputenc}
\usepackage[british]{babel}
\usepackage[T1]{fontenc}

\usepackage[hyperindex]{hyperref}

\newtheorem{defi}{\textbf{Definition}}[section]
\newtheorem{prop}[defi]{\textbf{Proposition}}
\newtheorem{thm}[defi]{\textbf{Theorem}}
\newtheorem{lem}[defi]{\textbf{Lemme}}

\theoremstyle{definition}
\newtheorem{expl}[defi]{\textbf{Example}}

\newtheorem{rmk}[defi]{\textbf{Remark}}

\newcommand{\sign}{\text{sign}}
\newcommand{\diag}{\text{diag}}

\newcommand{\opnorm}[1]{\left\lvert\hspace{-1 pt}\left\lvert\hspace{-1 pt}\left\lvert#1\right\lvert\hspace{-1 pt}\right\lvert\hspace{-1 pt}\right\lvert}
\newcommand{\Lrm}{\mathrm{L}}

\newcommand{\Wrm}{\mathrm{W}}

\def\keywords{\textbf{Key words. }}
\def\MSC{\textbf{2020 Mathematics Subject Classifications. }}
\def\abstractcmd{\textbf{Abstract. }}

\counterwithin{figure}{section}
\numberwithin{equation}{section}
\usepackage{csquotes}
\usepackage[style=alphabetic,sorting=nyt,date=year,url=false,maxalphanames=5,maxnames=5,giveninits=true,isbn=false,eprint=false,backref=true]{biblatex}
\title{\textbf{A Becker-Döring model with injection and irreversible fragmentation}}
\author{\textsc{Simon LOIN}\thanks{Université de Picardie Jules Verne, LAMFA, CNRS UMR 7352, 33 rue Saint-Leu, 80039, Amiens, France. E-mail: simon.loin@u-picardie.fr}}
\date{}

\begin{document}

\maketitle

\abstractcmd{We introduce and analyse a variant of the Becker-Döring equations that models the growth of clusters through the gain or loss of monomers. Motivated 
by enzymatic reactions in biology, this model incorporates irreversible fragmentation and monomers injection. We establish the 
well-posedness of the equations under suitable conditions on the kinetic rates. Then, as in the Becker-Döring equations, 
we distinguish two cases for the long time behaviour of our solution, however the distinction is made from the constant rate injection 
of monomers. While under strong fragmentation rate the system may exhibit infinite steady-states, we prove for low injection rate and 
moderate fragmentation the solution converges locally exponentially fast to the steady-state. Finally, we present 
an efficient scheme that preserves the asymptotic and allows fast computation by sub-sampling the clusters.
} \newline

 \newline

\MSC{34A34; 34D20; 92-10.}

\tableofcontents

\section{Introduction}

In their article \cite{becker_doring_1935}, Becker and Döring provided one of the earliest descriptions of particle growth in the theory of nucleation 
from supersaturated vapour, which subsequently gave the name to the model. This model describes the growth and decay of clusters, consisting 
of identical monomers, only by the addition and removal of monomers. The modern formulation of the equations seems to go back to Burton 
\cite{burton_nucleation_1977} to study condensations phenomena at different pressures, and was popularized among mathematicians by Penrose 
and Lebowitz \cite{penrose_lebowitz_1979}. Since then, the model has been applied in a variety of fields, including but not limited 
to physics, chemistry and biology. Initial mathematical results were proved by Ball, Carr and Penrose \cite{ball_becker-doring_nodate}, 
\cite{ball_carr_asymptotic_1988}, then the mathematical aspect of these equations have been studied in detail. Well-posedness and 
many aspects of the long time behaviour of solutions are understood \cite{jabin_rate_2003}, 
\cite{canizo_exponential_2013}, \cite{canizo_uniform_2019} and many other works, as well as the emergence of phase transition 
\cite{PenroseMetastable}. However, there are still open questions, see e.g. \cite{hingant_yvinec_2017} and \cite{wattis_introduction_2006}.

We also point out the fact that the Becker-Döring equations are a specific case of general discrete coagulation-fragmentation equations. 
Therefore classical results, mainly on the well-posedness, can be applied. First mathematical results were proved by Ball and Carr \cite{ball_discrete_1990},
Carr \cite{carr_asymptotic_1992}, Carr and da Costa \cite{carr_asymptotic_1994} and da Costa \cite{da_costa_existence_1995} on the 
well-posedness and the asymptotic behaviour of solutions. Another proof of existence was found by Laurençot \cite{laurencot_discrete_2002} 
extending existence result of Carr and da Costa \cite{carr_asymptotic_1994}. \newline

In this work, we study a variant of Becker-Döring equations.
We have in mind specific applications to biology with polymerisation of biomolecules such as fibrin clots formation. It is thus natural 
to consider \emph{open systems} in which monomers arise from a reaction cascades (e.g. fibrinogen conversion to fibrin protein) and a 
fragmentation process undergoes an irreversible process such as fibrin digestion by plasmin \cite{longstaff_basic_2015}.
In our version of Becker-Döring equations, we consider a model where injection of monomers is constant which may represents 
synthesis of fibrin monomer; and the detached monomer from fragmentation is not able to go through the coagulation 
process again. The reactions taking place are given by
\[\left\{
    \begin{aligned}
        (1) + (i) &\xrightarrow[]{a_{i}} (i+1), ~i\geq 1, \\
        (i) &\xrightarrow[]{b_i} (i-1), ~i\geq 3, \\
        (2) &\xrightarrow[]{b_2} \emptyset, \\
        \emptyset &\xrightarrow[]{\lambda} (1),
    \end{aligned}
    \right.
\]
where $(i)$ represents the concentration of clusters of size $i$. The non-negative numbers $a_i$ and $b_i$ denote respectively the 
coagulation and fragmentation coefficients. It is important to note that the $(2) \xrightarrow[]{b_2} \emptyset$ reaction is the 
result of a specific choice that was made, namely that a 2-polymer should disappear and not form a single monomer. The degradation 
of the 2-polymer is presumed to occur in a manner that both monomers are degraded. We believe that with slight adaptations our 
results holds with considering the reaction $(2) \xrightarrow[]{b_2} (1)$.

The infinite system of ordinary differential equations associates to the reaction scheme is: 
\begin{equation}
    \left\{
    \begin{aligned}
    \frac{d}{dt}C_1(t) &= \lambda - \sum_{j=1}^{+\infty}a_{j}C_1(t)C_j(t) - a_1C_1(t)^2, &  \\ 
    \frac{d}{dt}C_i(t) &= J_{i-1}(t) - J_i(t),  & i\geq 2,
    \end{aligned}
    \right.
    \label{systeme equations BD-Depoly}
\end{equation}
where, for all $i\geq 1$,
\begin{equation}
    J_i(t) = a_iC_1(t)C_i(t) - b_{i+1}C_{i+1}(t).
    \label{flux J BD-Depoly}
\end{equation}
The unknowns are the functions $C = (C_i(t))_{i\geq 1}$ which depend on time $t\geq 0$ and where, for each $i\in\mathbb{N}^\ast$, 
$C_i(t)$ denotes the concentration of $i$-particles clusters per unit of volume at time $t$. The quantity $J_i$ does not  
represent the rate of change of a reversible reaction as in the Becker-Döring equations, but for convenience we keep this notation. \newline

Some work already exists on models close to Becker-Döring and close to ours. The first modification was to add monomers injection, 
it appears in submonolayer epitaxial growth \cite{bales_dynamics_1994}, \cite{amar_kinetics_1996}, 
in nucleation theory for thin film growth \cite{ratsch_nucleation_2003}.
The first mathematical work was done by Blackman and Marshall \cite{blackman_coagulation_1994}, they discussed scaling behaviour 
and growth exponent for a constant rate source; then by Wattis \cite{wattis_similarity_2004} where he studied self-similar 
behaviour of solutions for a time-dependent monomer input and size-independent rate coefficients. He has already shown that with a slight 
modification, the long time behaviour of solutions may become complex and depend strongly on the production rate. 
We also mention the works of da Costa and al. \cite{da_costa_long-time_2006}, \cite{da_costa_convergence_2007} and \cite{da_costa_rates_2016}
which, with constant rate coefficients, complement some of the formal results in \cite{blackman_coagulation_1994} and \cite{wattis_similarity_2004}.
Recently Niethammer and al. \cite{niethammer_oscillations_2022} introduce a depletion term representing clusters removal. As our 
model the chemical reaction network is open, it then becomes unclear whether long-time convergence towards steady-state holds. In their 
case, they provide evidence for the persistence of oscillations in time.
We also mention the work of Bolton and Wattis \cite{bolton_beckerdoring_2004}, in which they considered a model with injection, 
competition and inhibition.

There exists other modifications of the Becker-Döring equations but differs to ours, for example the work of Doumic and 
al. \cite{doumic_mezache_bi-monomeric_2019} introduce a bi-monomeric model explaining oscillations. We also mention the work of Wattis 
\cite{wattis_modified_2009} in which the fragmentation rates depends on the total number of clusters present in the system; and 
the works of Laurençot and Wrzosek \cite{laurencot_becker-doring_1998_1} and \cite{laurencot_becker-doring_1998_2} in which they 
added space diffusion. 

Some literature exists on continuous models with mass loss, more specifically on coagulation and fragmentation equations with discrete 
and continuous mass loss. The first mathematical work was done by Edwards \cite{edwards_rate_1990}, and a more recent work was done 
by Blair and al. \cite{blair_coagulation_2007}. A part of the book of Banasiak, Lamb and Laurençot \cite{banasiak_analytic_2019} 
treat those equations using semigroups techniques. Different authors \cite{baird_mixed_2019} and \cite{ali_note_2025} considered 
a hybrid model of fragmentation still with discrete and continuous mass loss. \newline

In this work, we start by proving well-posedness of (\ref{systeme equations BD-Depoly}). The proof of existence follows the one 
of Laurençot \cite{laurencot_discrete_2002}  which, as for the proof of Ball and al. \cite{ball_becker-doring_nodate} and \cite{ball_discrete_1990}, 
relies on a truncated system and compactness arguments to obtain the limit. Uniqueness is more classical and follows the initial 
work of Ball and Carr \cite{ball_discrete_1990}. As in the Becker-Döring equations two distinct cases appears, however here it 
does not depend on the initial mass, since our system does not preserve it, but instead on the production rate $\lambda$. There 
is a threshold $\lambda_s$ which dictates significant change in the dynamic. In the sub-critical case $\lambda \leq \lambda_s$, 
it may occur that
\begin{itemize}
    \item either there exists a unique steady-state for $b_i \leq ia_i$,
    \item or there exists an infinite number of steady-states for $b_i > i^\nu a_i$ with $\nu > 1$.
\end{itemize}
We prove that if $\lambda$ is sufficiently small and $a_i < b_i \leq i$ then the steady-state is locally exponentially 
asymptotically stable. This corresponds to $b_i$ strong enough but not too much (moderate), this conditions will be made precise 
in the section \ref{section loc exp stab}.
In the super-critical case $\lambda > \lambda_s$, contrary to the Becker-Döring equations, we do not found a Lyapunov functional, 
thus we were not able to prove a result in the super-critical case; the question remains open, but we expect some self-similar 
behaviour of the solution \cite{wattis_similarity_2004}.
We end this work with a numerical scheme. We develop, on the conservative truncation, 
a well-balanced and coarse-grained scheme, which consist of sub-sampling the clusters at some given size, preserving the asymptotics. 
The scheme is based on the flux approximation scheme of Ducan and 
Soheili \cite{duncan_approximating_2001} and the well-balanced scheme of Goudon and Monasse \cite{goudon_fokker-planck_2020}.
This scheme can be use on the classical Becker-Döring equations.

\section{Main results}

In this section, we introduce some notations, hypotheses, and state the main results. Details and proofs are 
presented in their respected section.
Before starting our existence and uniqueness result, we provide the definition of a solution to (\ref{systeme equations BD-Depoly}).

\begin{defi}
    Let $T\in (0,+\infty]$. A solution $C = (C_i)_{i\geq 1}$ of (\ref{systeme equations BD-Depoly}) on $[0,T)$ is a sequence of non-negative functions satisfying the following conditions for all $i\geq 1$, $t\in [0,T)$,
    \begin{enumerate}[label=(\roman*)]
        \item $C_i \in \mathcal{C}^0([0,t))$, $~\displaystyle\sum_{j=1}^{+\infty} a_{j}C_j \in \Lrm^1(0,t),$
        \item and there holds
            \begin{equation*}
                \left\{
                \begin{aligned}
                C_1(t) &= C_1(0) + \lambda t - \int_0^t C_1(s)\sum_{j=1}^{+\infty}a_{j}C_j(s) - a_1C_1(s)^2 ds , & i=1 \\ 
                C_i(t) &= C_i(0) + \int_0^t \left[ J_{i-1}(s) - J_i(s)\right]ds,  & i\geq 2.
                \end{aligned}
                \right.
            \end{equation*} 
    \end{enumerate}
    \label{defi solution BD-Depoly}
\end{defi}

We start by stating the weak form of (\ref{systeme equations BD-Depoly}), which is very useful to compute moments, or to find some uniform bounds.
\begin{prop}
    Let $C$ a solution of (\ref{systeme equations BD-Depoly}). For all compactly supported sequences $(\varphi_i)_{i\geq 1}$, we have
    \begin{equation}
        \frac{d}{dt}\sum_{i=1}^{+\infty} \varphi_iC_i = \lambda\varphi_1 - \varphi_2b_2C_2 + \sum_{i=1}^{+\infty} \left[\varphi_{i+1}-\varphi_i-\varphi_1\right]a_iC_1C_i + \sum_{i=3}^{+\infty} \left[\varphi_{i-1}-\varphi_i\right]b_iC_i.
        \label{weak form 1}
    \end{equation}
    We can write it under another form as follows
    \begin{equation}
        \frac{d}{dt}\sum_{i=1}^{+\infty} \varphi_iC_i = \varphi_1\left(\lambda -b_2C_2 - \sum_{i=2}^{+\infty} b_iC_i\right) + \sum_{i=1}^{+\infty} \left[\varphi_{i+1}-\varphi_i-\varphi_1\right]\left(a_iC_1C_1-b_{i+1}C_{i+1}\right).
        \label{weak form 2}
    \end{equation}
\end{prop}

\bigskip

We define the following Banach spaces in which solutions will lie. Let $\alpha \geq 0$, we define the following subspaces of 
$\ell^1(\mathbb{R})$,
\begin{equation*}
    X_\alpha := \left\{ x = (x_k)_{k\geq 1} \in \mathbb{R}^\mathbb{N} ~:~ \|x\|_{X_\alpha} := \sum_{k=1}^{+\infty} k^\alpha|x_k| < +\infty \right\},
\end{equation*}
and we denote $X_\alpha^+$ its positive cone. Some of these spaces have physical meaning, for example the norm in $X_0 = \ell^1$ is 
proportional to the total number of clusters, or the norm of $X_1$ is proportional to the mass of all the clusters in our system. 
Some properties of these spaces can be found in the review of da Costa \cite[Section 2.1]{DaCostaOverview}.

\subsection{Well-posedness}

We state the main results on existence and uniqueness of solutions under some assumptions on the kinetic coefficients, mainly on 
the coagulation ones, and on the initial data. 
Let $\alpha \in [0,1]$.
\begin{enumerate}[label=(H\arabic*)]
    \item There exists $a \geq 0$ such that $0 \leq a_{i} \leq a i^\alpha$ and $b_i \geq 0$ for all $i\geq 1$. \label{borne a_i sous lineaire et positivite b_i}
\end{enumerate}
\vspace{0.1cm}

In the classical proof of existence from Ball, Carr and Penrose \cite{ball_becker-doring_nodate} or 
for the general coagulation-fragmentation equations from \cite{ball_discrete_1990} or \cite{laurencot_discrete_2002} the natural 
space for the existence is $X_1$ since the equations are mass preserving. Here equations (\ref{systeme equations BD-Depoly}) do not 
preserve mass, therefore we can consider a larger space for the existence, in which mass does not necessarily have meaning or may blow 
up. We have the following existence results in $X_\alpha$ for $\alpha \in [0,1]$.
\begin{thm}[\textbf{Existence and moment spreading}]
    Let $C^{init} \in X_\alpha^+$. Under the hypothesis \ref{borne a_i sous lineaire et positivite b_i}, 
    there exists at least one solution $C \in \mathcal{C}^0([0,+\infty),X_\alpha^+)$ to (\ref{systeme equations BD-Depoly}) with initial 
    data $C(0) = C^{init}$. Moreover, if there exists $\mu \geq \alpha$ such that 
    \begin{equation*}
        \sum_{i=1}^{+\infty} i^\mu C_i^{init} < +\infty,
    \end{equation*}
    then any solution satisfies for all $T>0$
    \begin{equation*}
        \sup_{t\in [0,T]}\sum_{i=1}^{+\infty} i^\mu C_i(t) < +\infty.
    \end{equation*}
    \label{thm existence X_alpha}
\end{thm}

\noindent We actually prove a more general result including more than just power law moments but a wilder class, see section \ref{section existence}.
This theorem is inspired from \cite[Theorem 2.5]{laurencot_discrete_2002}, and it generalises \cite[Theorem 3.3]{carr_asymptotic_1994} 
which proves moment spreading in coagulation-fragmentation equations.

\begin{thm}[\textbf{Partial uniqueness}]
    Let $C^{init}\in X_\alpha^+$ and $T\in (0,+\infty]$. Under the hypothesis \ref{borne a_i sous lineaire et positivite b_i},
    there exists at most one solution $C$ to (\ref{systeme equations BD-Depoly}) on $[0,T)$ with initial condition $C(0)=C^{init}$ such that
    \begin{equation}
        C \in \Lrm^\infty([0,t],X_\alpha^+) \text{ and } \sum_{i=1}^{+\infty} i^{2\alpha} C_i \in \Lrm^1(0,t) \text{ for all } t \in (0,T).
        \label{sum iphi_iC_i unicité}
    \end{equation}
    \label{thm unicité X_alpha}
\end{thm}

\begin{thm}[\textbf{Well-posedness}]
    Let $C^{init}\in X_{2\alpha}^+$ and $T\in (0,+\infty]$. Under the hypothesis \ref{borne a_i sous lineaire et positivite b_i},
    there exists a unique solution $C$ to (\ref{systeme equations BD-Depoly}) on $[0,T)$ with initial condition $C(0)=C^{init}$. Moreover, 
    this solution satisfies
    \begin{equation*}
        C \in \mathcal{C}^0([0,t],X_\alpha^+) \text{  and  } \sum_{i=1}^{+\infty} i^{2\alpha} C_i \in \Lrm^1(0,t) \text{ for all } t \in (0,T).
    \end{equation*}
    \label{thm wellposedness}
\end{thm}

\noindent \textbf{Theorem \ref{thm wellposedness}} is an immediate consequence of \textbf{Theorem \ref{thm unicité X_alpha}} and 
\textbf{Theorem \ref{thm existence X_alpha}}. The partial uniqueness result \textbf{Theorem \ref{thm unicité X_alpha}} is 
inspired by \cite[Theorem 4.2]{ball_discrete_1990} and \cite[Proposition 5.1]{laurencot_discrete_2002} and state uniqueness under 
moments control.  
Laurençot and Mischler \cite[Theorem 2.1]{laurencot_beckerdoring_2002} gave another proof of uniqueness, on the Becker-Döring 
equations, without the assumption on moments control, but with an assumption on coefficients $(b_i)_{i\geq 1}$. Their proof 
is based on estimates of distribution tails and the mass conservation.
The proof of \textbf{Theorem \ref{thm unicité X_alpha}} is given in section \ref{section uniqueness}.
\newline

In the remainder, we will always assume that the kinetics coefficients satisfies the condition \ref{borne a_i sous lineaire et positivite b_i} 
for some suitable $\alpha \in [0,1]$.

\subsection{Long time behaviour}

\subsubsection{Uniform-in-time bound}

The first uniform-in-time bound is an upper bound on $C_1(t)$. We obtain directly from the equation on $C_1$ that for all $t\in [0,+\infty)$
we have
\begin{equation}
    C_1(t) \leq \kappa(\lambda,C^{init}) := \max\left\{C_1^{init},\sqrt{\frac{\lambda}{2a_1}}\right\}.
    \label{UIT bound of C_1}
\end{equation}

We prove an upper uniform-in-time bound on exponential moments under strong fragmentation, and this will allow us to prove the local 
exponential stability of a steady-state. The following proposition is inspired from a similar result in \cite[Lemma 3.3]{jabin_rate_2003}.
\begin{prop}[\textbf{Exponential moments uniform-in-time bound}]
    Assume that the kinetics coefficients satisfies the following conditions:
    \begin{equation}
        \forall i\geq 1, ~a_i \geq \underline{a} >0.
        \label{condition coefficient UIT exponential moment 1}
    \end{equation}
    Let $\lambda \geq 0$ and $C^{init}$ small enough such that 
    \begin{equation}
        \kappa(\lambda,C^{init}) < \inf_{i\geq 1} \frac{b_i}{a_i},
        \label{condition coefficient UIT exponential moment 2}
    \end{equation}
    and there exists  $\nu > 0$ such that
    \begin{equation}
        \sum_{i=1}^{+\infty} e^{\nu i}C_i^{init} < +\infty.
        \label{UIT bound exponential moment init}
    \end{equation}
    Then there exists a constant $\Xi_1 > 0$ such that for all $t\geq 0$,
    \begin{equation}
        \sum_{i=1}^{+\infty} e^{\nu i}C_i(t) \leq \Xi_1,
        \label{UIT bound exponential moment}
    \end{equation}
    where $\Xi_1$ depends only on $\lambda$, $(a_i)_{i\geq 1}$, $(b_i)_{i\geq 1}$, $\nu$ and initial data.
    \label{prop UIT exponential moment}
\end{prop}

\noindent We also prove, under some hypotheses on the kinetic coefficients, an upper bound on the mass and a lower bound on $C_1$.
Details and proofs are presented in section \ref{subsection UIT}.

\subsubsection{Steady-states}

We define the detailed balance coefficients $Q_i$ by
\begin{equation}
    Q_1 := 1,~~ Q_i := \prod\limits_{j=2}^i \frac{a_{j-1}}{b_j}, \text{ pour } i\geq 2.
    \label{defi Q_i}
\end{equation}
These coefficients play a significant role in the long-term behaviour of the solutions. It is from these coefficients that two 
significant quantities are defined. The first quantity is the \textit{critical monomer density}, denoted $z_s$, it is defined as the radius of 
convergence of the power series $\sum a_iQ_iz^i$. This quantity is also called \textit{monomer saturation density}, hence the subscript.
The second quantity is the \textit{critical production threshold}, denoted $\lambda_s$, it is defined as follows:
\begin{equation}
    \lambda_s := \sup_{z\leq z_s} \sum_{j=1}^{+\infty} a_jQ_jz^{j+1} + a_1 z^2.
    \label{defi lambda_s}
\end{equation}
We emphasize that both $z_s$ and $\lambda_s$ are completely determined by the coefficients $a_i$ and $b_i$. \newline

\begin{rmk}
    As the function $z\mapsto a_jQ_jz^{j+1}$ is increasing, the critical production threshold can be rewritten, when the series 
    converges, as follows
    \begin{equation*}
        \lambda_s =  \sum_{j=1}^{+\infty} a_jQ_jz_s^{j+1} + a_1 z_s^2.
    \end{equation*}
\end{rmk}

In the following, we add two hypotheses that can be readily verified in the case of constant coefficients, certain power laws, or 
Niethammer's coefficients \cite[(1.5) and (1.6)]{niethammer_evolution_2003}. 
These conditions are as follows.
\begin{enumerate}[label=(H\arabic*)] \addtocounter{enumi}{1}
    \item There exists $\underline{a}>0$ and $\underline{b}>0$ such that $\inf\limits_{i\geq 1}a_i=\underline{a}$ and $\inf\limits_{i\geq 1}b_i=\underline{b}$, \label{(H1)}
    \item There exists $l>0$ such that $\lim\limits_{i\to +\infty}\dfrac{Q_{i+1}}{Q_i}=\dfrac{1}{l z_s}$, \label{(H2)}
    \item $\lim\limits_{i\to +\infty}\dfrac{a_{i+1}}{a_i} = \lim\limits_{i\to +\infty}\dfrac{b_{i+1}}{b_i} = l$. \label{(H2bis)}
\end{enumerate}

\begin{rmk}
    Hypotheses \ref{(H2)} and \ref{(H2bis)} implies that $\lim\limits_{i\to +\infty}\frac{b_i}{a_i} = z_s$. \newline
    \label{rmk limite bi/ai}
\end{rmk}

Whether or not a steady-state exists depends on the critical production rate $\lambda_s$ and the production rate $\lambda$, and more 
particularly on their comparison. By steady-state we understand a constant solution $C$, in particular, 
$\sum_{i=1}^{+\infty} a_iC_i <\infty$.

\begin{thm}[\textbf{Steady-states}]
    Under the assumptions \ref{(H1)}-\ref{(H2bis)}, we have the two following possibilities:
    \begin{enumerate}
        \item (Sub-critical) If $\lambda \leq \lambda_s$, then there exists a unique $z\leq z_s$ verifying
        \begin{equation}
            \lambda = \sum_{j=1}^{+\infty} a_jQ_jz^{j+1} + a_1z^2
            \label{definition z sous-critique}
        \end{equation}
        such that $(Q_iz^i)_{i\geq 1}$ is a steady-state. Moreover, if $b_i\leq ia_i$ for $i\gg 1$ then $(Q_iz^i)_{i\geq 1}$ is the unique
        steady-state or if $b_i > i^\nu a_i$ for $i\gg 1$ and $\nu > 1$ then there is an infinity of steady-state.
        \item (Super-critical) Else $\lambda > \lambda_s$, and then there exists no steady-state.
    \end{enumerate}
    \label{thm equilibre}
\end{thm}

\subsubsection{Local exponential stability} \label{section loc exp stab}

The objective is to determine the long-term behaviour of the solution. We do not find a Lyapunov functional, consequently we used the 
linearised system of equations. We analyse the spectral properties of the linearised operator, which can be represented as a perturbation
of the one studied in the work of Cañizo and Lods \cite{canizo_exponential_2013}. \newline

In what follows, we assume the additional hypothesis:
\begin{enumerate}[label=(H\arabic*)]
    \setcounter{enumi}{4}
    \item $\exists ~\beta \in [0,1], \exists ~b >0, ~\forall i\geq 1, ~b_i\leq bi^\beta$. \label{(H4)} 
\end{enumerate}

\begin{rmk}
    We point out that \textbf{Remark \ref{rmk limite bi/ai}} with $z_s \neq 0$ and \ref{(H1)} implies that $\inf_{i\geq 1}\frac{b_i}{a_i}$ is a 
    positive constant.
\end{rmk}

\begin{thm}[\textbf{Local exponential convergence}]
    Assume \ref{borne a_i sous lineaire et positivite b_i}-\ref{(H4)}. 
    Let $\lambda < \lambda_s$ and $C^{init}$ such that 
    \begin{equation}
        \kappa(\lambda,C^{init}) < \inf_{i\geq 1}\frac{b_i}{a_i}
        \label{equilibrium sub-critical condition kappa bi/ai}
    \end{equation}
    and that (\ref{UIT bound exponential moment init}) holds for some $\nu >0$. Let $C=(C_i)_{i\geq 1}$ a solution of (\ref{systeme equations BD-Depoly}) 
    with initial condition $C^{init}$. Take $z$ satisfying (\ref{definition z sous-critique}). Then there exists some $\eta\in(0,\nu)$, some $K,\varepsilon>0$ 
    and some $\mu_\ast>0$ such that if
    \begin{equation}
        \frac{8b}{\sqrt{z}}\sqrt{\sum_{i=1}^{+\infty}i^{2\beta}Q_iz^i}\sup_{k\geq 1} \left(\sum_{j=k+1}^{+\infty} Q_jz^j\right)\left(\sum_{j=1}^k\frac{1}{a_jQ_jz^j}\right)< 1
        \label{equilibrium sub-critical condition P-norm spectral gap}
    \end{equation}
    and
    \begin{equation}
        \sum_{i=1}^{+\infty} \exp(\eta i)|C_i^{init}-Q_iz^i| \leq \varepsilon,
    \end{equation}
    then
    \begin{equation}
        \sum_{i=1}^{+\infty}\exp(\eta i)|C_i(t)-Q_iz^i| \leq K\exp(-\mu_\ast t),~ \forall t\geq 0.
        \label{local expo convergence}
    \end{equation}
    \label{thm local expo convergence}
\end{thm}

\begin{rmk}
    Condition (\ref{equilibrium sub-critical condition P-norm spectral gap}) is a sufficient condition. 
    One might ask whether this condition is ever met. To address this, we give the following examples.
\end{rmk}

\begin{expl}[\textbf{Constant coefficients}]
    Constant kinetics coefficients $a_i = a$ and $b_i = b$ satisfy hypotheses \ref{borne a_i sous lineaire et positivite b_i}-\ref{(H4)} 
    for some $a,b$ well choose. The condition (\ref{equilibrium sub-critical condition kappa bi/ai}) becomes either $C_1^{init} < \frac{b}{a}$ 
    or $\lambda < 2\frac{b^2}{a}$.
    The condition (\ref{equilibrium sub-critical condition P-norm spectral gap}) becomes
    \begin{equation*}
        \frac{(z_s-z)^{5/2}}{8zz_s^{5/2}} >1.
    \end{equation*}
    Since the left-hand side goes to $+\infty$ when $z$ goes to $0$, this condition holds for $\lambda$ small enough. Indeed, 
    recalling the definition of $z$ (\ref{definition z sous-critique}), we see that the smaller $\lambda$ is, the smaller $z$ is too.
    \label{expl constant coeff}
\end{expl}

\begin{expl}[\textbf{Linear coefficients}]
    Linear coefficients $a_i = ai$ and $b_i = bi$ satisfy hypotheses \ref{borne a_i sous lineaire et positivite b_i}-\ref{(H4)} for some $a,b$ well choose. The condition (\ref{equilibrium sub-critical condition P-norm spectral gap}) becomes for $z$ small enough
    \begin{equation*}
        \frac{32z}{\sqrt{z_s}(z_s-z)}<1.
    \end{equation*}
    As before, since the left-hand side goes to $0$ when $z$ goes to $0$, this condition holds for $\lambda$ small enough.
    \label{expl linear coeff}
\end{expl}

\noindent Details for \textbf{Example \ref{expl constant coeff}} and \textbf{Example \ref{expl linear coeff}} are respectively presented 
in \textbf{Appendix \ref{annexe constant coeff}} and \textbf{\ref{annexe linear coeff}}.

\subsection{Numerical simulations}

To simulate our equations numerically, we need to truncate our system, since we obviously can't keep an infinite number of 
ordinary differential equations. We therefore have different ways of truncating our system. Generally speaking, for 
coagulation-fragmentation equations, there are two ‘natural’ truncations: conservative truncation and non-conservative 
truncation. Conservative truncation means that at the size of truncation, we are not allowing this maximum size to coagulate 
and no aggregates of this size are the result of depolymerisation, we have the following system of equations:
\begin{equation}
    \left\{
    \begin{aligned}
    \frac{d}{dt}C_1(t) &= \lambda - \sum_{j=1}^{n-1}a_{j}C_1(t)C_j(t) - a_1C_1(t)^2, & i=1, \\ 
    \frac{d}{dt}C_i(t) &= J_{i-1}(t) - J_i(t),  & 1<i<n, \\
    \frac{d}{dt}C_n(t) &= J_{n-1}(t), & i=n.
    \end{aligned}
    \right.
    \label{systeme equations BD-Depoly tronc conserv}
\end{equation}
This truncation corresponds to $J_i = 0$ for $i\geq n$. Another truncation exists taking $J_n = a_nC_1C_n$ referred as 
non-conservative truncation which means that at the size of truncation, we are allowing this maximum size to coagulate 
and as before no aggregates of this size are the result of depolymerisation. \newline

To numerically simulate these previous system we can just use standard ODE schemes (e.g. Runge-Kutta 4), however choosing a large truncation $n$ may involve 
very long computation time. Therefore, the goal is to develop a coarse-grain scheme, consisting of sub-sampling the clusters at 
some given size, preserving the asymptotics of the system. 

The main idea is to see our truncation (\ref{systeme equations BD-Depoly tronc conserv})
as some discretisation of a particular PDE for some size step $\Delta x$, in which the size of the aggregates $x$ is a continuous 
variable. Let $n\geq 2$ an integer, the aforementioned PDE is the following:
\begin{equation}
    \frac{\partial C(t,x)}{\partial t} = -\frac{\partial J(t,x)}{\partial x}, \text{ in } [0,T]\times (1,n),
    \label{PDE type focker planck}
\end{equation}
with
\begin{equation}
    J(t,x) = -a(x)Q(x)C(t,1)^{x+1/2}\frac{\partial}{\partial x}\left(\frac{C(t,x)}{Q(x)C(t,1)^x}\right),
    \label{PDE flux J}
\end{equation}
where the functions $a(x),b(x)$ and $Q(x)$ are the step functions of the discrete quantities $a_i,b_i$ and $Q_i$. 
Formally, (\ref{PDE flux J}) is obtained through the following discretisation with $\Delta x = 1$
\begin{multline}
    J_i(t) = \frac{1}{\Delta x}\left(a_{i-1}C_1C_{i-1} - a_{i-1}\frac{Q_{i-1}}{Q_i}C_1\right) \\
    = -a_{i-1}Q_{i-1}C_1^i \frac{1}{\Delta x}\left(\frac{C_i}{Q_iC_1^i}-\frac{C_{i-1}}{Q_{i-1}C_1^{i-1}}\right)
        \approx J\left(t,i-\frac{1}{2}\right).
    \label{PDE flux J discretisation delta x = 1}
\end{multline}
To which we add boundary conditions, the right one depends on the truncation that we are considering and the left one comes from $C_1$. 
This idea is, in particular, presented by Duncan and Soheili \cite{duncan_approximating_2001}. Thus, taking another discretisation 
of the PDE, for example with a non-uniform mesh, and simulate the PDE on this mesh should approach the desired solution. 
To that end, we are going to use the idea presented by Goudon and Monasse \cite{goudon_fokker-planck_2020}. 
We mention the work of Bolton and Wattis \cite{bolton_general_2002}, in which they develop a coarse-grain scheme on the 
Becker-Döring equations.
Details are presented in section \ref{section : numerical simualations}.

\subsection{Open questions}

We highlight some questions that remain open. The uniqueness of a solution without the moment control, 
as in \cite[Theorem 2.1]{laurencot_beckerdoring_2002}, is open since no quantity is preserved. The long-time behaviour of 
the solution, except for our local stability result, remains open even in the sub-critical case, and especially when 
there exists an infinity of steady-states. The behaviour of the solution in the super-critical case is even more unclear 
since no steady-state exist.

\section{Existence} \label{section existence}

To prove \textbf{Theorem \ref{thm existence X_alpha}} we follow the idea of Laurençot from \cite{laurencot_discrete_2002}. To start, 
we introduce three sets of functions which are going to be useful later
\[ \mathcal{K}_1 := \left\{ U \in C^1([0,+\infty))\cap \Wrm^{2,\infty}_{loc}(0,+\infty)~:~U\geq 0,~U \text{ convex},~U(0)=0,~U'(0)\geq 0,\text{ et } U' \text{ concave} \right\}, \]
\[ \mathcal{K}_{1,\infty} := \left\{ U \in \mathcal{K}_1~:~\lim_{r\to +\infty} U'(r) = \lim_{r\to +\infty} \frac{U(r)}{r} = +\infty \right\}, \]
\[ \mathcal{K}_2 := \left\{ U \in \mathcal{C}^2([0,+\infty),\mathbb{R}_+) ~:~ U(0)=U'(0)=0~, U' \text{ convex satisfying the } \Delta_2-\text{condition}\right\},\]
with the $\Delta_2-$condition being
\begin{equation} 
    U'(2r)\leq k_UU'(r), ~ \forall r\geq 0, \text{ and some } k_U>0.
    \label{delta2 condition}
\end{equation}

\begin{rmk}
    The function $x\mapsto x^m$ belongs to $\mathcal{K}_1$ for $m\in [1,2]$, belongs to $\mathcal{K}_{1,\infty}$ 
    for $m\in (1,2]$ and belongs to $\mathcal{K}_2$ for $m\geq 2$. \newline
\end{rmk}

As mention in the beginning, we prove a more general result than \textbf{Theorem \ref{thm existence X_alpha}}, which consider 
propagation of general moments. It reads as follows:

\begin{thm}[\textbf{Existence and moment spreading}]
    Let $C^{init} \in X_\alpha^+$. Under the hypothesis \ref{borne a_i sous lineaire et positivite b_i}, 
    and assume that there exists $U\in \mathcal{K}_1\cup\mathcal{K}_2$ such that
    \begin{equation}
        \sum_{i=1}^{+\infty} U(i^\alpha)C_i^{init} < +\infty.
        \label{moment U init fini}
    \end{equation}
    Then there exists at least one solution $C \in \mathcal{C}^0([0,+\infty),X_\alpha^+)$ to (\ref{systeme equations BD-Depoly}) with initial data $C(0) = C^{init}$, and for each $T\in (0,+\infty),$
    \begin{equation}
        \sup_{t\in [0,T]}\sum_{i=1}^{+\infty} U(i^\alpha)C_i(t) < +\infty.
        \label{moment U fini}
    \end{equation}
    \label{thm existence X_alpha + moment U}
\end{thm}

\noindent \textbf{Theorem \ref{thm existence X_alpha}} follows by taking $U(x)=x^m$ for $m\geq 1$.

\subsection{Preliminary results}\label{subsection existence prelim}

As in many works on similar coagulation-fragmentation equations (e.g. \cite{ball_becker-doring_nodate}, \cite{ball_discrete_1990}, 
\cite{carr_asymptotic_1992} or \cite{carr_asymptotic_1994}), the existence of a solution to (\ref{systeme equations BD-Depoly}) 
satisfying $C(0)=C^{init}$ comes by taking the limit of solutions to a finite number system of differential 
equations obtained by truncating our system of equations (\ref{systeme equations BD-Depoly}). More precisely, let $n\geq 3$, 
and recall the system of $n$ ordinary differential equations (\ref{systeme equations BD-Depoly tronc conserv}), to which we 
still add the initial condition 
\begin{equation}
    C_i^n(0) = C_i^{init} \geq 0, ~ 1\leq i\leq n.
    \label{systeme equations BD-Depoly tronqué cond initiale}
\end{equation}

\noindent As in the infinite system, we can write the weak form of this system of equations. It reads as follows
\begin{equation}
    \frac{d}{dt}\sum_{i=1}^{n} \varphi_iC_i = \lambda\varphi_1 - \varphi_2b_2C_2 + \sum_{i=1}^{n-1} \left[\varphi_{i+1}-\varphi_i-\varphi_1\right]a_iC_1C_i + \sum_{i=3}^{n} \left[\varphi_{i-1}-\varphi_i\right]b_iC_i.
    \label{weak form 1 tronqué}
\end{equation}
\vspace{0.2cm}

To start, we prove the well-posedness of the truncated system of equations (\ref{systeme equations BD-Depoly tronc conserv}).
\begin{prop}
    Let $n\geq 3$ and $C^{init}\in X_\alpha^+$. There exists a unique solution $C^n \in \mathcal{C}^1([0,+\infty),\mathbb{R}^n)$ to
    (\ref{systeme equations BD-Depoly tronc conserv}) satisfying $C_i^n(0)=C_i^{init}$ for $i=1,\ldots,n$. This solution satisfies
    $C_i^n(t) \geq 0$ for all $t\geq 0$ and $i=1,\ldots,n$ and 
    \begin{equation}
        C_i^n(t) \leq \sum_{i=1}^n i^\alpha C_i^n(t) \leq \lambda T + \|C^{init}\|_{X_\alpha^+}
        \label{borne C_i^n par borne cond init Xalpha}
    \end{equation}
    for all $T\in (0,+\infty)$ and $i=1,\ldots,n$.
    \label{prop existence solution systeme tronqué}
\end{prop}

\begin{proof}
    The proof is a straight forward application of the Cauchy-Lipschitz theorem. The non-negativity follows form \cite[Theorem 7.1 (Positivity)]{ducrot_griette}.
    The fact that 
    \[ \frac{d}{dt}\sum_{i=1}^n C_i^n(t) \leq \lambda \]
    entails non-explosion in finite time. Moreover, using (\ref{weak form 1 tronqué}) with $\varphi_i = i^\alpha$ and since $\alpha \in [0,1]$,
    we obtain
    \begin{align*}
        \frac{d}{dt}\sum_{i=1}^n i^\alpha C_i^n &= \lambda - 2^\alpha b_2C_2^n + \sum_{i=1}^{n-1}\left[(i+1)^\alpha-i^\alpha-1\right]a_iC_i^nC_1^n + \sum_{i=3}^n\left[(i-1)^\alpha-i^\alpha\right]b_iC_i^n \leq \lambda,
    \end{align*}
    which entails (\ref{borne C_i^n par borne cond init Xalpha}).
\end{proof}

Our objective is now to show that we can extract from the sequence $(C^n)_{n\geq 1}$, even if we take $C_i^n = 0$ for $i>n$, 
a sub-sequence which converges, and the limit is a solution, i.e. satisfies the \textbf{Definition \ref{defi solution BD-Depoly}}. 
In order to do that, we need to establish the following two lemmas.

\begin{lem}
    Let $T\in(0,+\infty)$ and $U\in\mathcal{K}_1\cup\mathcal{K}_2$. There exists a constant $\gamma_T$ depending only on $a, U, C^{init}$ and $T$ such that for all $n\geq 3$, and $t \in [0,T]$, we have
    \begin{equation}
        \sum_{i=1}^n U(i^\alpha)C_i^n(t) \leq \gamma_T \sum_{i=1}^n U(i^\alpha)C_i^{init} + \lambda U(1)T\gamma_T
    \end{equation}
    \label{lemme existence 1}
\end{lem}

\begin{proof}
    Let $n\geq 3$ and $M_U^n(t) := \sum_{i=1}^n U(i^\alpha)C_i^n(t)$. We have for $t\geq 0$
    \begin{align*}
        \frac{d}{dt}M_U^n(t) &\leq \lambda U(1) - U(2^\alpha)b_2C_2^n(t) + \sum_{i=3}^{n} \left[U((i-1)^\alpha)-U(i^\alpha)\right]b_{i}C_{i}^n(t) \\
        &~~+ \sum_{i=1}^{n-1}\left[U((i+1)^\alpha)-U(i^\alpha)-U(1)\right]a_{i}C_i^n(t)C_1^n(t)
    \end{align*}
    Noticing, for all $i\geq 1$ and $\alpha \in [0,1]$, $ (i+1)^\alpha \leq i^\alpha + 1$,
    and as $U \in \mathcal{K}_1\cup\mathcal{K}_2$, $U$ is increasing, thus $U((i+1)^\alpha) \leq U(i^\alpha + 1)$ for all $i\geq 1$.
    Moreover, from \cite[Lemma 3.2]{laurencot_discrete_2002} there exists a constant $m_U$ depending only on $U$ such that for all $i\geq 1$
    \begin{equation}
        (i^\alpha + 1)(U(i^\alpha + 1)-U(i^\alpha)-U(1)) \leq m_U(i^\alpha U(1) + U(i^\alpha)).
        \label{inequalite lemme laurencot}
    \end{equation}
    Therefore by \ref{borne a_i sous lineaire et positivite b_i} and (\ref{inequalite lemme laurencot}), we obtain
    \begin{align*}
        \frac{d}{dt}M_U^n(t) &\leq \lambda U(1) - U(2^\alpha)b_2C_2^n(t) + \sum_{i=3}^{n} \left[U((i-1)^\alpha)-U(i^\alpha)\right]b_{i}C_{i}^n(t) \\
        &~~+ am_U\sum_{i=1}^{n-1}(i^\alpha U(1)+U(i^\alpha))C_i^n(t)C_1^n(t),
    \end{align*}
    with $a$ respectively $m_U$, the constant from \ref{borne a_i sous lineaire et positivite b_i} respectively (\ref{inequalite lemme laurencot}).
    As, $U$ is increasing, $U((i-1)^\alpha)-U(i^\alpha) \leq 0$ for all $i\geq 1$, thus
    \begin{align*}
        \frac{d}{dt}M_U^n(t) &\leq \lambda U(1) + 2 a m_U M_\alpha^n(t) M_U^n(t) \\
        &\leq \lambda U(1) + 2a m_U(\lambda T + \|C^{init}\|_{X_\alpha})M_U^n(t)
    \end{align*}
    Therefore, using Grönwall's lemma, we conclude
    \begin{equation*}
        M_U^n(t) \leq  (M_U^n(0) + \lambda U(1)T)e^{2a m_U(\lambda T + \|C^{init}\|_{X_\alpha})T}
    \end{equation*}
    which ends the proof. \newline
\end{proof}

\begin{lem}
    Let $T \in (0,+\infty)$. For each $i\geq 1$, there exists a constant $\Gamma_i(T)$ depending only on $a, i, \|C^{init}\|_{X_\alpha}$ et $T$ such that, for all $n\geq i$,
    \begin{equation}
        \left\| \frac{d}{dt}C_i^n \right\|_{\Lrm^1(0,T)} \leq \Gamma_i(T).
    \label{borne dérivée solution tronquée}
    \end{equation}
    \label{lemme existence 2}
\end{lem}

\begin{proof}
    From \textbf{Proposition \ref{prop existence solution systeme tronqué}}, there exists a constant $K_T$ such that 
    $\left| C_i^n(t)\right| \leq K_T$ for all $i\geq 1$. Hence, for $i\geq 2$, $\left|\dfrac{d}{dt}C_i^n(t)\right|$ is uniformly bounded in $\Lrm^1(0,T)$ 
    according to $n$, thanks to equations (\ref{systeme equations BD-Depoly tronc conserv}). It remains to bound $\left|\dfrac{d}{dt}C_1^n(t)\right|$, 
    which follows from the bound
    \begin{equation*}
        \int_0^T \sum_{i=1}^{n-1}a_iC_1^nC_i^n \leq K_Ta \int_0^T\sum_{i=1}^n i^\alpha C_i^n \leq TaK_T^2.
    \end{equation*}
\end{proof}

\subsection{Proof of existence theorem}

We are now in a position to prove \textbf{Theorem \ref{thm existence X_alpha + moment U}}. 
However, let's not forget a refined version of the de La Vallée-Poussin theorem for integrable functions from \cite[Proposition I.1.1]{ChauHoan}, 
which we are going to need. That is the key point of Laurençot's proof in \cite[Theorem 2.3 and Theorem 2.5]{laurencot_discrete_2002} that enables him to 
provide a different kind of proof without the need for cleverly decomposing the sums as in \cite[Theorem 2.4 and Theorem 2.5]{ball_discrete_1990}.

\begin{thm}[\textbf{de La Vallée-Poussin}]
    Let $\left(\Omega,\mathcal{B},\mu\right)$ a measured space, and let $w\in \Lrm^1(\Omega,\mathcal{B},\mu)$. Then there exists a function $V \in \mathcal{K}_{1,\infty}$ such that
    \[ V(|w|) \in \Lrm^1(\Omega,\mathcal{B},\mu). \]
    \label{thm vallee-poussin amelio}
\end{thm}

\begin{proof}[Proof of Theorem \ref{thm existence X_alpha + moment U}]
    First we want to apply \textbf{Theorem \ref{thm vallee-poussin amelio}} with $\Omega = \mathbb{N}^\ast$. To that end, we define the measure $\mu$ by 
    \[ \mu(I) = \sum_{i\in I} C_i^{init},~~ I \subset \mathbb{N}\backslash\left\{0\right\}. \]
    The condition $C^{init} \in X_\alpha^+$ ensures that the function $x\mapsto x^\alpha$ belongs to $\Lrm^1(\Omega,\mathcal{B},\mu)$. Therefore, using \textbf{Theorem \ref{thm vallee-poussin amelio}}, there exists a function $U_0 \in \mathcal{K}_{1,\infty}$ such that $x\mapsto U_0(x^\alpha)$ belongs to $\Lrm^1(\Omega,\mathcal{B},\mu)$, i.e.
    \begin{equation}
        \mathcal{U}_0 := \sum_{i=1}^{+\infty} U_0(i^\alpha)C_i^{init} < \infty.
        \label{condition U_0 alpha}
    \end{equation}  \newline
    Moreover, for $n\geq i$,
    \begin{equation}
        C_i^n(t) \leq M_\alpha^n(t) \leq \lambda T + \|C^{init}\|_{X_\alpha},\text{ on } [0,T],
        \label{majoration C_i^n}
    \end{equation} 
    and by \textbf{Lemma \ref{lemme existence 2}}, the sequence $(C_i^n)_{n\geq i}$ is bounded in $\Lrm^\infty(0,T)\cap \Wrm^{1,1}(0,T)$ 
    for each $i\geq 1$ and $T\in (0,\infty)$. Therefore, using Helly's theorem \cite[Théorème 5 (p. 372)]{Kolmogorov}, there exists a subsequence of $(C_i^n)_{n\geq i}$, still denoted $(C_i^n)_{n\geq i}$, and a sequence $C$ of bounded variation functions such that 
    \begin{equation}
        \lim_{n\to\infty} C_i^n(t) = C_i(t),~~\forall i\geq 1,~\forall t\geq 0.
        \label{limite suite C_i^n}
    \end{equation}
    Thus, using (\ref{majoration C_i^n}) and (\ref{limite suite C_i^n}), we have $C\in X_\alpha^+$, with
    \begin{equation}
       \|C(t)\|_{X_\alpha^+} \leq \lambda T + \|C^{init}\|_{X_\alpha}.
       \label{borne norme C X_alpha}
    \end{equation}
    \vspace{0.1cm}
    
    It now remains to show that the potential candidate $C$ is indeed a solution of (\ref{systeme equations BD-Depoly}), 
    satisfying $C(0) = C^{init}$. To that end, we need to show two estimates. First, as $U_0 \in \mathcal{K}_{1,\infty}$, 
    using \textbf{Lemma \ref{lemme existence 1}} and (\ref{condition U_0 alpha}), we have for all $T\geq 0$, and $n\geq 3$,
    \begin{equation}
        \sum_{i=1}^n U_0(i^\alpha)C_i^n(t) \leq \xi(T),~~ t\in [0,T],
        \label{borne somme U(i^alpha)C_i}
    \end{equation}
    where the constant $\xi(T)$ depend only on $a, C^{init}, U_0$ and $T$.
   
    Let $T\in (0,+\infty)$ and $m\geq 2$, considering $C_i^n = 0$ if $i\geq n$, the previous estimate becomes, for $n\geq m+1$,
    \begin{equation*}
        \sum_{i=1}^m U_0(i^\alpha)C_i^n(t) \leq \xi(T),~~ t\in [0,T].
    \end{equation*}
    Thanks to (\ref{limite suite C_i^n}), we can pass to the limit when $n\to +\infty$ in the previous estimates, thus it stays 
    true substituting $C_i^n$ by $C_i$. We then pass to the limit on $m\to +\infty$, and we obtain
    \begin{equation}
        \sum_{i=1}^{+\infty} U_0(i^\alpha)C_i(t) \leq \xi(T),~~ t\in [0,T].
        \label{borne somme inf U(i^alpha)C_i}
    \end{equation}
    Secondly, using \ref{borne a_i sous lineaire et positivite b_i} and (\ref{borne norme C X_alpha}), we obtain for all $i\geq 1$
    \begin{equation}
       \sum_{j=1}^{+\infty}a_{j}C_j \in \Lrm^1(0,T).
       \label{demo continuite}
    \end{equation}
    Now, we state that, for all $i\geq 1$, we have
    \begin{align}
       &\lim_{n\to +\infty}\left\| C_1^n\sum_{j=1}^{n-1} a_{j}C_j^n - C_1 \sum_{j=1}^{+\infty} a_{j}C_j \right\|_{\Lrm^1(0,T)} = 0, \label{assertion 1}\\
       &\lim_{n\to +\infty}\left\| a_iC_i^nC_1^n - a_iC_iC_1 \right\|_{\Lrm^1(0,T)} = 0, \label{assertion 2}\\
       &\lim_{n\to +\infty} \left\| b_iC_i^n - b_iC_i \right\|_{\Lrm^1(0,T)} = 0. \label{assertion 3}
    \end{align}
    The last two assertions (\ref{assertion 2}) and (\ref{assertion 3}) come from (\ref{limite suite C_i^n}), (\ref{borne norme C X_alpha}), (\ref{borne C_i^n par borne cond init Xalpha}) and from Lebesgue-dominated convergence theorem.
    It remains to prove the first assertion (\ref{assertion 1}). Let $i\geq 1$ fixed, and $m\geq 2$. As before, using (\ref{limite suite C_i^n}), (\ref{borne norme C X_alpha}) and (\ref{borne C_i^n par borne cond init Xalpha}) and Lebesgue-dominated convergence theorem, we obtain for $m\leq n-1$,
    \begin{equation}
       \lim_{n\to +\infty} \left\| \sum_{j=1}^m a_{j}(C_1^nC_j^n-C_1C_j) \right\|_{\Lrm^1(0,T)} = 0.
       \label{demo partie 1}
    \end{equation}
    Moreover, for $n-1\geq m+1$, and using \ref{borne a_i sous lineaire et positivite b_i}, (\ref{borne norme C X_alpha}) and (\ref{borne somme inf U(i^alpha)C_i}) we obtain
    \begin{align}
       \left\|\sum_{j=m+1}^{+\infty}a_{j}C_1C_j\right\|_{\Lrm^1(0,T)} &\leq a(\lambda T + \|C^{init}\|_{X_\alpha})\left\|\sum_{j=m+1}^{+\infty}j^\alpha C_j\right\|_{\Lrm^1(0,T)} \nonumber \\
       &\leq a(\lambda T + \|C^{init}\|_{X_\alpha}) \sup_{j\geq m+1} \frac{j^\alpha}{U_0(j^\alpha)}\left\|\sum_{j=m+1}^{+\infty} U_0(j^\alpha) C_j\right\|_{\Lrm^1(0,T)}  \nonumber \\
       &\leq a(\lambda T + \|C^{init}\|_{X_\alpha})T\xi(T) \sup_{j\geq m+1} \frac{j^\alpha}{U_0(j^\alpha)}. \label{demo partie 2}
    \end{align}
    Likewise, using \ref{borne a_i sous lineaire et positivite b_i}, (\ref{borne C_i^n par borne cond init Xalpha}) and (\ref{borne somme U(i^alpha)C_i}) we obtain
        \begin{align}
        \left\|\sum_{j=m+1}^{n-i}a_{j}C_1^nC_j^n\right\|_{\Lrm^1(0,T)} &\leq a(\lambda T + \|C^{init}\|_{X_\alpha})\left\|\sum_{j=m+1}^{+\infty}j^\alpha C_j^n\right\|_{\Lrm^1(0,T)} \nonumber \\
        &\leq a(\lambda T + \|C^{init}\|_{X_\alpha}) \sup_{j\geq m+1} \frac{j^\alpha}{U_0(j^\alpha)}\left\|\sum_{j=m+1}^{+\infty} U_0(j^\alpha) C_j^n\right\|_{\Lrm^1(0,T)}  \nonumber \\
        &\leq a(\lambda T + \|C^{init}\|_{X_\alpha})T\xi(T) \sup_{j\geq m+1} \frac{j^\alpha}{U_0(j^\alpha)}. \label{demo partie 3}
    \end{align}
    Therefore, combining (\ref{demo partie 1})-(\ref{demo partie 3}), we obtain that for all $m\geq 2$,
    \begin{equation*}
        \limsup_{n\to +\infty} \left\| \sum_{j=1}^{n-1} a_{j}C_1^nC_j^n - \sum_{j=1}^{+\infty} a_{j}C_1C_j \right\|_{\Lrm^1(0,T)} \leq 2a(\lambda T + \|C^{init}\|_{X_\alpha})T\xi(T) \sup_{j\geq m+1}\frac{j^\alpha}{U_0(j^\alpha)}.
    \end{equation*}
    As $U_0$ belongs to $\mathcal{K}_{1,\infty}$, we have that the right-hand side of the above inequality tends to $0$ when $m\to +\infty$, thus we obtain (\ref{assertion 1}). \newline
   
   Now, thanks to (\ref{limite suite C_i^n}) and (\ref{assertion 1})-(\ref{assertion 3}), passing to the limit in $n$ we obtain $(ii)$ 
   of the \textbf{Definition \ref{defi solution BD-Depoly}}. By (\ref{demo continuite}), we recover the continuity of $C_i$ for all 
   $i\geq 1$. Therefore, we have shown that $C=(C_i)_{i\geq 1} \in X_\alpha^+$ is a solution of (\ref{systeme equations BD-Depoly}) 
   on $[0,T]$ satisfying $C(0)=C^{init}$. Note that the continuity of $C(\cdot)$ comes from the uniform convergence thanks to 
   Dini's theorem of the sequence of continuous functions $\left(\sum\limits_{i=1}^n i^\alpha C_i(\cdot)\right)_{n\geq 1}$ on $[0,T]$. \newline

    To conclude, we only need to show that the solution constructed in the previous proof enjoys the additional property (\ref{moment U fini}). 
    As $U \in \mathcal{K}_1\cup\mathcal{K}_2$, (\ref{moment U fini}) follows from \textbf{Lemma \ref{lemme existence 1}} and the convergence 
    (\ref{limite suite C_i^n}).
\end{proof}

\section{Uniqueness} \label{section uniqueness}

We first need some estimates on a solution which satisfy (\ref{sum iphi_iC_i unicité}). These are as follows:

\begin{lem}
    A solution $C$ of (\ref{systeme equations BD-Depoly}) on $[0,T)$ which satisfy the condition (\ref{sum iphi_iC_i unicité}) 
    also satisfies, for all $t\in (0,T)$,
    \begin{equation}
        \lim_{n\to +\infty} \int_0^t n^\alpha a_n C_n(s)C_1(s)ds = 0,
        \label{lemme unicite assertion 1}
    \end{equation}
    \begin{equation}
        \lim_{n\to +\infty} \int_0^t \sum_{i=n}^{+\infty} a_iC_i(s)C_1(s)ds = 0.
        \label{lemme unicite assertion 2}
    \end{equation}
    \label{lemme technique unicité}
\end{lem}

\begin{proof}
    Let $n\geq 2$ and $t\in (0,T)$, recalling \ref{borne a_i sous lineaire et positivite b_i} we have
    \begin{equation*}
        \int_0^t n^\alpha a_n C_1C_n ds \leq \int_0^t \sum_{i=n}^{+\infty} i^\alpha a_iC_i C_1 ds \leq a\int_0^t \sum_{i=n}^{+\infty} i^{2\alpha} C_i C_1 ds .
    \end{equation*}
    and 
    \begin{equation*}
        \int_0^t \sum_{i=n}^{+\infty} a_iC_i C_1 ds \leq a\int_0^t \sum_{i=n}^{+\infty} i^{2\alpha} C_i C_1 ds .
    \end{equation*}
    Therefore, both assertions (\ref{lemme unicite assertion 1}) and (\ref{lemme unicite assertion 2}) comes form Lebesgue-dominated convergence theorem using 
    (\ref{UIT bound of C_1}) and (\ref{sum iphi_iC_i unicité}). 
\end{proof}

\begin{proof}[Proof of Theorem \ref{thm unicité X_alpha}]
    Let $C=(C_i)_{i\geq 1}$ and $D=(D_i)_{i\geq 1}$ denote two solutions of (\ref{systeme equations BD-Depoly}) on $[0,T)$ with $C(0)=D(0)=C^{init}$ which satisfy (\ref{sum iphi_iC_i unicité}). For $i\geq 1$, defines
    \[ z_i := C_i-D_i, \text{  and  } g_i:=i^\alpha~\sign(z_i),\]
    where $\sign$ denotes the sign function, i.e. $\sign(x)=\frac{x}{|x|}$ for $x \in \mathbb{R}^*$ and $\sign(0)=0$. Now, we fix $n\geq 2$, and $t\in (0,T)$. The goal is to show that $\|z(t)\|_{X_\alpha} = 0$, which will entail the desired uniqueness result. \newline

    Note that if $f \in AC(\mathbb{R})$ then $|f| \in AC(\mathbb{R})$, and therefore
    \[ \frac{d}{dt}|f(t)| = \sign(f(t))\frac{df}{dt}(t) \text{ a.e. in } \mathbb{R}. \]
    Since $C$ and $D$ are absolute continuous functions, $z$ is also an absolute continuous function and therefore using the integral 
    form of (\ref{systeme equations BD-Depoly}), we obtain
    \begin{align*}
    \sum_{i=1}^n i^\alpha|z_i(t)| = \sum_{i=1}^n g_iz_i(t) &= -\int_0^t \sum_{j=1}^{+\infty} g_1 a_j [C_1(s)C_j(s) - D_1(s)D_j(s)] + g_1a_1[C_1(s)^2-D_1(s)^2]ds  \\
    &~~ + \sum_{i=2}^n \int_0^t g_i a_{i-1}\left[C_{i-1}(s)C_1(s)-D_{i-1}(s)D_1(s)\right]ds \\
    &~~ -\sum_{i=2}^n \int_0^t g_i a_{i}\left[C_{i}(s)C_1(s)-D_{i}(s)D_1(s)\right]ds \\
    &~~ + \sum_{i=2}^n \left( -\int_0^t g_ib_i\left[C_i(s)-D_i(s)\right]ds + \int_0^t g_ib_{i+1}\left[C_{i+1}(s)-D_{i+1}(s)\right]ds \right) \\
    &= \int_0^t \sum_{i=1}^{n-1} (g_{i+1}-g_i-g_1)a_i[C_i(s)C_1(s) - D_i(s)D_1(s)] ds \\
    & ~~ - \int_0^t g_{n}a_n[C_1(s)C_n(s)-D_1(s)D_n(s)] ds \\
    &~~ - \sum_{i=n}^{+\infty} \int_0^t g_1 a_{i}\left[C_{i}(s)C_1(s)-D_{i}(s)D_1(s)\right]ds \\
    &~~ + \sum_{i=2}^n \left( -\int_0^t g_ib_i\left[C_i(s)-D_i(s)\right]ds + \int_0^t g_ib_{i+1}\left[C_{i+1}(s)-D_{i+1}(s)\right]ds \right).
    \end{align*}
    However, $C_{i}(s)C_1(s)-D_{i}(s)D_1(s) = C_i(s)z_1(s)+D_1(s)z_i(s)$, therefore
    \begin{align*}
        \sum_{i=1}^n i^\alpha|z_i(t)| &\leq \int_0^t \sum_{i=1}^{n-1} (g_{i+1}-g_i-g_1)a_i[C_i(s)z_1(s) + z_i(s)D_1(s)] ds \\
        & ~~ - \int_0^t g_{n}a_n[C_1(s)C_n(s)-D_1(s)D_n(s)] ds \\
        &~~ - \sum_{i=n}^{+\infty} \int_0^t g_1 a_{i}\left[C_{i}(s)C_1(s)-D_{i}(s)D_1(s)\right]ds \\
        &~~ + \sum_{i=2}^n \left( -\int_0^t g_ib_i\left[C_i(s)-D_i(s)\right]ds + \int_0^t g_ib_{i+1}\left[C_{i+1}(s)-D_{i+1}(s)\right]ds \right).
    \end{align*}
    As $\sign(x) \leq 1$ for all $x\in \mathbb{R}$, we have for all $i,j\geq 1$,
    \begin{equation}
    \begin{aligned}
        (g_{i+j}-g_i-g_j)z_j &= ((i+j)^\alpha\sign(z_{i+j})-i^\alpha\sign(z_i)-j^\alpha\sign(z_j))z_j \\
        &= ((i+j)^\alpha\sign(z_{i+j}z_j)-i^\alpha\sign(z_iz_j)-j^\alpha)|z_j| \\
        &\leq ((i+j)^\alpha + i^\alpha - j^\alpha)|z_j| \\
        &\leq 2i^\alpha|z_j|.
    \end{aligned}
    \label{demo unicite inegalite g_{i+1}}
    \end{equation}
    Moreover, using \ref{borne a_i sous lineaire et positivite b_i} we obtain
    \begin{align*}
        \sum_{i=2}^n \left( -\int_0^t g_ib_iz_i(s)ds + \int_0^t g_ib_{i+1}z_{i+1}(s)ds \right) &= \int_0^t \sum_{i=2}^{n-1} g_{i-1}b_iz_i(s) - \sum_{i=2}^ng_ib_iz_i(s) ds \\
        &= -\int_0^t \sum_{i=2}^{n-1} (g_{i}-g_{i-1})b_iz_i(s) ds - \int_0^t \underbrace{n^\alpha b_n|z_n(s)|}_{\geq 0}ds. 
    \end{align*}
    However, $(g_i-g_{i-1})z_i = (i^\alpha-(i-1)^\alpha\sign(z_{i-1}z_i))|z_i| \geq 0$, this entails that
    \begin{equation}
        \sum_{i=2}^n \left( -\int_0^t g_ib_iz_i(s)ds + \int_0^t g_ib_{i+1}z_{i+1}(s)ds \right) \leq 0. 
    \label{demo unicite positivite terme}
    \end{equation}
    Therefore, using (\ref{demo unicite inegalite g_{i+1}}) and (\ref{demo unicite positivite terme}), we have
    \begin{equation}
        \begin{aligned}
        \sum_{i=1}^n i^\alpha|z_i(t)| &\leq \int_0^t \sum_{i=1}^{n-1} a_i i^\alpha |z_1(s)|C_i(s) +\sum_{i=1}^{n-1} a_i |z_i(s)|D_1(s)ds \\
        & ~~ - \int_0^t g_{n}a_n[C_1(s)C_n(s)-D_1(s)D_n(s)] ds \\ 
        &~~ - \int_0^t \sum_{i=n}^{+\infty} g_1 a_{i}\left[C_{i}(s)C_1(s)-D_{i}(s)D_1(s)\right]ds.
        \label{preuve unicite assertion 0}
        \end{aligned}
    \end{equation}
    We now treat one by one the three terms on the right-hand side of the above inequality. \newline
    \textit{First term :} Using \ref{borne a_i sous lineaire et positivite b_i}, we have
    \begin{align*}
        \int_0^t \sum_{i=1}^{n-1} a_i i^\alpha |z_1(s)|C_i(s) ds &+ \int_0^t \sum_{i=1}^{n-1} a_i |z_i(s)|D_1(s)ds \\
        &\leq \int_0^t \sum_{i=1}^n i^\alpha a_i C_i(s) \sum_{i=1}^n |z_i(s)| + \sum_{i=1}^n a_i|z_i(s)| \sum_{i=1}^n i^\alpha D_i(s) ds \\
        &\leq \int_0^t \left( \sum_{i=1}^{n} i^\alpha a_i C_i(s) + a \sum_{i=1}^n i^\alpha D_i(s) \right)\sum_{i=1}^n i^\alpha |z_i(s)| ds.
    \end{align*}
    Using (\ref{sum iphi_iC_i unicité}) and Lebesgue-dominated convergence theorem, we obtain
    \begin{multline}
        \lim_{n\to +\infty} \int_0^t \sum_{i=1}^{n-1} a_i i^\alpha |z_1(s)|C_i(s) +\sum_{i=1}^{n-1} a_i |z_i(s)|D_1(s)ds \\ \leq \int_0^t \left( \sum_{i=1}^{+\infty} i^\alpha a_i C_i(s) + a (\lambda T + \|C^{init}\|_{X_\alpha}) \right)\sum_{i=1}^{+\infty} i^\alpha |z_i(s)| ds.
        \label{preuve unicite assertion 1}
    \end{multline}
    \textit{Second term :} As $-g_n[C_1C_n-D_1D_n] \leq n^\alpha [C_1C_n+D_1D_n]$, we have
    \begin{align*}
        - \int_0^t g_{n}a_n[C_1(s)C_n(s)-D_1(s)D_n(s)] ds \leq \int_0^t n^\alpha a_n C_1(s)C_n(s) ds + \int_0^t n^\alpha a_n D_1(s)D_n(s) ds.
    \end{align*}
    Using \ref{borne a_i sous lineaire et positivite b_i} and (\ref{lemme unicite assertion 1}) from \textbf{Lemma \ref{lemme technique unicité}}, we obtain
    \begin{equation}
        \lim_{n\to +\infty} - \int_0^t g_{n}a_n[C_1(s)C_n(s)-D_1(s)D_n(s)] ds = 0.
        \label{preuve unicite assertion 2}
    \end{equation}
    \textit{Third term :} We have 
    \[ - \int_0^t \sum_{i=n}^{+\infty} g_1 a_{i}\left[C_{i}(s)C_1(s)-D_{i}(s)D_1(s)\right]ds \leq \int_0^t \sum_{i=n}^{+\infty} a_iC_1(s)C_i(s) ds + \int_0^t \sum_{i=n}^{+\infty} a_iD_1(s)D_i(s)ds. \]
    Therefore, using (\ref{lemme unicite assertion 2}) from \textbf{Lemma \ref{lemme technique unicité}}, we obtain
    \begin{equation}
        \lim_{n\to +\infty} - \int_0^t \sum_{i=n}^{+\infty} g_1 a_{i}\left[C_{i}(s)C_1(s)-D_{i}(s)D_1(s)\right]ds = 0.
        \label{preuve unicite assertion 3}
    \end{equation}
    Passing to the limit $n\to +\infty$ in (\ref{preuve unicite assertion 0}), the previous estimations (\ref{preuve unicite assertion 1})-(\ref{preuve unicite assertion 3}) entails that
    \begin{equation*}
        \sum_{i=1}^{+\infty} i^\alpha |z_i(t)| \leq \int_0^t \left( \sum_{i=1}^{+\infty} i^\alpha a_i C_i(s) + a (\lambda T + \|C^{init}\|_{X_\alpha}) \right)\sum_{i=1}^{+\infty} i^\alpha |z_i(s)| ds.
    \end{equation*}
    Then applying Grönwall's lemma, we obtain that for all $t\in [0,T)$,
    \begin{equation*}
        \sum_{i=1}^{+\infty} i^\alpha |z_i(t)| \leq 0,
    \end{equation*}
    which entails $C=D$, and ends the proof.
\end{proof}

\section{Long time behaviour}

\subsection{Uniform-in-time bound} \label{subsection UIT}

We start by proving \textbf{Proposition \ref{prop UIT exponential moment}}.
\begin{proof}[Proof of Proposition \ref{prop UIT exponential moment}]
    Let $M>1$, for convenience we rewrite condition (\ref{condition coefficient UIT exponential moment 2}) as 
    \begin{equation}
        \frac{b_i}{a_i} \geq M\kappa \text{ for all } i\geq 1.
        \label{proof UIT hyp}
    \end{equation}  
    Let $t\geq 0$ and $\mu \in (1,M)$. We have the following formal computations:
    \begin{align*}
        \frac{d}{dt}\sum_{i=1}^{+\infty}\mu^iC_i(t) &= \mu\left(\lambda - \sum_{j=1}^{+\infty}a_jC_j(t)C_1(t)-a_1C_1(t)^2\right) + \sum_{i=2}^{+\infty}\mu^i(J_{i-1}(t)-J_i(t)) \\
        &= \mu\left(\lambda - \sum_{j=1}^{+\infty}a_jC_j(t)C_1(t)-a_1C_1(t)^2\right) + (\mu - 1)\sum_{i=1}^{+\infty}\mu^iJ_i + \mu(a_1C_1(t)^2-b_2C_2(t)) \\
        &\leq \mu\lambda - \mu a_1C_1(t)^2 + (\mu-1)\sum_{i=1}^{+\infty}\mu^i(a_iC_i(t)C_1(t)-b_{i+1}C_{i+1}(t)) + \mu(a_1C_1(t)^2-b_2C_2(t)) \\
        &= \mu\lambda + (\mu-1)\sum_{i=1}^{+\infty}\mu^{i-1}a_i\left(\mu C_1(t)-\frac{b_i}{a_i}\right)C_i(t) +(\mu-1)b_1C_1(t) - \mu b_2C_2(t) \\
        &= \mu\lambda + \frac{(\mu-1)}{\mu}\sum_{i=1}^{+\infty}\mu^{i}a_i\left(\mu C_1(t)-\frac{b_i}{a_i}\right)C_i(t) +(\mu-1)b_1C_1(t) - \mu b_2C_2(t) \\
        &\leq \mu\lambda +(\mu-1)b_1C_1(t) + \frac{(\mu-1)}{\mu}\sum_{i=1}^{+\infty}\mu^{i}a_i\left(\mu C_1(t)-\frac{b_i}{a_i}\right)C_i(t).
    \end{align*}
    Then using (\ref{UIT bound of C_1}) and (\ref{proof UIT hyp}), and denoting $\varepsilon = (M-\mu)\kappa > 0$, we have
    \begin{align*}
        \frac{d}{dt}\sum_{i=1}^{+\infty}\mu^iC_i(t) &\leq \mu\lambda + (\mu-1)b_1\kappa - \varepsilon \frac{\mu-1}{\mu}\sum_{i=1}^{+\infty}a_i\mu^iC_i(t) \\
        &\leq \mu\lambda + (\mu-1)b_1\kappa - \varepsilon \frac{\mu-1}{\mu}\underline{a}\sum_{i=1}^{+\infty}\mu^iC_i(t).
    \end{align*}
    Therefore, using Grönwall's lemma, we obtain for all $t\geq 0$ and $\mu\in(1,M)$,
    \begin{equation*}
        \sum_{i=1}^{+\infty}\mu^iC_i(t) \leq \sum_{i=1}^{+\infty}\mu^iC_i^{init} 
        + \frac{\lambda\mu+(\mu-1)b_1\kappa}{\varepsilon\frac{\mu-1}{\mu}\underline{a}} := \Xi_1,
    \end{equation*}
    where $\Xi_1$ depends only on $\lambda$, $(a_i)_{i\geq 1}$, $(b_i)_{i\geq 1}$, $\mu$ and the initial condition. Taking $\nu = \ln(\mu)$, 
    we have the desired estimation (\ref{UIT bound exponential moment}). \newline

    The rigorous proof can be left to the reader, as it consists in reproducing these estimates on the truncated system 
    (\ref{systeme equations BD-Depoly tronc conserv}), (\ref{systeme equations BD-Depoly tronqué cond initiale}) and then
    letting $n\to +\infty$ using $C_i^n \to C_i$ and $\sum_{i=1}^n \mu^iC_i^n \to \sum_{i=1}^{+\infty} \mu^iC_i$.

\end{proof}

As mention in the beginning, we have an upper bound on the mass and lower bound on $C_1(t)$.
\begin{prop}[\textbf{Mass uniform-in-time bound}]
    Let $b>0$. Assume that $C^{init} \in X_1^+$ and that the kinetics coefficients satisfies the following conditions:
    \begin{equation}
        \forall i\geq 1,~ b_i\geq bi.
        \label{condition coefficient UIT masse}
    \end{equation}
    Then, for all $t\geq 0$,
    \begin{equation}
        \sum_{i=1}^{+\infty} iC_i(t) \leq \max\left\{\sum_{i=1}^{+\infty} iC_i^{init}, \frac{\lambda + b_1\kappa}{b}\right\}.
        \label{UIT bound of mass}
    \end{equation}
    \label{prop UIT mass}
\end{prop}

\begin{proof}[Proof of Proposition \ref{prop UIT mass}]
    Recalling (\ref{weak form 1}), and using the assumption (\ref{condition coefficient UIT masse}) and the uniform-in-time bound of $C_1$ (\ref{UIT bound of C_1}), we have for all $t\geq 0$,
    \begin{align*}
        \frac{d}{dt}\sum_{i=1}^{+\infty}iC_i(t) &= \lambda -b_2C_2 - \sum_{i=2}^{+\infty}b_iC_i \\
        &\leq \lambda -2bC_2+b_1C_1-b\sum_{i=1}^{+\infty}iC_i \\
        &\leq \lambda + b_1\kappa - b\sum_{i=1}^{+\infty} iC_i.
    \end{align*}
    Therefore, for all $t\geq 0$,
    \begin{equation*}
        \sum_{i=1}^{+\infty}iC_i(t) \leq \max\left\{\sum_{i=1}^{+\infty} iC_i^{init}, \frac{\lambda + b_1\kappa}{b}\right\}.
    \end{equation*}
\end{proof}

\begin{prop}[\textbf{$C_1$ uniform-in-time positivity}]
    Let $a,b>0$. Assume that the kinetics coefficients satisfies the following conditions:
    \begin{equation}
        \forall i\geq 1,~ a_i\leq ai ~\text{ and }~ b_i\geq bi.
        \label{condition coefficient UIT C1 positivité}
    \end{equation}
    Then, there exists a constant $\eta >0$ and $t_1 \geq 0$ such that for all $t\geq t_1$ we have
    \begin{equation}
        C_1(t) \geq \eta.
    \end{equation}
    \label{prop UIT C1 minoration}
\end{prop}

\begin{proof}[Proof of Proposition \ref{prop UIT C1 minoration}]
    First, using (\ref{condition coefficient UIT C1 positivité}) and the uniform-in-time bound of the mass (\ref{UIT bound of mass}) we have for all $t\geq 0$,
    \begin{equation*}
        \sum_{i=1}^{+\infty}a_iC_i(t) \leq a \sum_{i=1}^{+\infty}iC_i(t) \leq a \max\left\{\sum_{i=1}^{+\infty}iC_i^{init},\frac{\lambda+b_1\kappa}{b}\right\}.
    \end{equation*}
    Then, for all $t\geq 0$, we have
    \begin{align*}
        \frac{d}{dt}C_1(t) &= \lambda - \sum_{j=1}^{+\infty}a_jC_j(t)C_1(t) - a_1C_1(t)^2 \\
        &\geq \lambda - a \max\left\{\sum_{i=1}^{+\infty}iC_i^{init},\frac{\lambda+b_1\kappa}{b}\right\}C_1(t) - a_1C_1(t)^2.
    \end{align*}
    Therefore, as long as $C_1 \leq \min\left\{ 1,a \max\left\{\sum\limits_{i=1}^{+\infty}iC_i^{init},\frac{\lambda+b_1\kappa}{b}\right\} +a \right\}$, we have
    \begin{align*}
        \frac{d}{dt}C_1(t) &\geq \lambda - \left[ a \max\left\{\sum_{i=1}^{+\infty}iC_i^{init},\frac{\lambda+b_1\kappa}{b}\right\} + a_1\right] C_1(t)\\
        &\geq \lambda - \left[ a \max\left\{\sum_{i=1}^{+\infty}iC_i^{init},\frac{\lambda+b_1\kappa}{b}\right\} + a\right] C_1(t) \geq 0.
    \end{align*}
    Hence, there exists $t_1 \geq 0$ such that for all $t\geq t_1$,
    \begin{equation*}
        C_1(t) \geq \frac{1}{2} \min\left\{ 1, a \max\left\{\sum_{i=1}^{+\infty}iC_i^{init},\frac{\lambda+b_1\kappa}{b}\right\} + a \right\} =: \eta.
    \end{equation*}
\end{proof}

\subsection{Steady-states}

In this subsection we prove \textbf{Theorem \ref{thm equilibre}}. Let $C=(C_i)_{i\geq 1}$ a constant solution. From 
(\ref{systeme equations BD-Depoly}), we obtain
\begin{equation*}
    J_{i-1} = J_i, ~\forall i \geq 2.
\end{equation*}
Therefore, there exists $J \in \mathbb{R}$ such that
\begin{equation}
    J_i = J, ~\forall i\geq 1.
    \label{equilibre flux J}
\end{equation}
We then have three possibilities: either $J=0$, or $J > 0$ or $J<0$. We immediately obtain that there is no steady-state with a positive 
flux $J>0$. Indeed, from \ref{borne a_i sous lineaire et positivite b_i} and the positivity of solutions, we have
\begin{equation*}
    a_iC_1C_i = J + b_{i+1}C_{i+1} \geq J.
\end{equation*}
Therefore,
\begin{equation*}
    \sum_{i=1}^{+\infty}a_iC_1C_i \geq \sum_{i=1}^{+\infty}J = +\infty,
\end{equation*}
which means that there is no steady state when $J>0$ due to the equation on $C_1$. \newline

The case $J=0$ gives the steady-state $(Q_iz^i)_{i\geq 1}$. Assuming that a steady-state exists, 
we denote it $C^{eq}$, the equation (\ref{equilibre flux J}) entails the following relation:
\begin{equation*}
    a_iC_1^{eq}C_i^{eq} = b_{i+1}C_{i+1}^{eq}, ~\forall i\geq 1.
\end{equation*}
Thus,
\begin{equation*}
    C_i^{eq} = \left(\prod_{j=2}^{i} \frac{a_{j-1}}{b_j}\right)(C_1^{eq})^i = Q_i(C_1^{eq})^i, ~\forall i\geq 1.
\end{equation*}
The steady-state is therefore entirely determined by $C_1^{eq}$, to determine this value we use the first equation of the system 
(\ref{systeme equations BD-Depoly}) which has not yet been used. We therefore have that $C_1^{eq}$ is a solution of the following 
equation in $z$
\begin{equation}
    \lambda =  \sum_{j=1}^{+\infty}a_{j}Q_jz^{j+1} + a_1 z^2.
    \label{equilibre Qizi C_1^(eq)}
\end{equation}
By comparing with the critical production threshold $\lambda_s$ defined in (\ref{defi lambda_s}), the equation (\ref{equilibre Qizi C_1^(eq)}) 
admits a unique solution when $\lambda \leq \lambda_s$, which verifies $z \leq z_s$. \newline

The last case is the negative flux $J<0$. Using equation (\ref{equilibre flux J}), we have 
\begin{equation}
    J(z) := a_izg_i(z) - b_{i+1}g_{i+1}(z), ~~ \forall i\geq 1,
    \label{equilibre BD-Depoly flux J(z) 2}
\end{equation}
where $z$ and functions $g_i(z)$ designate a constant solution. In other words, we search a steady-state
\[ C_i = g_i(z), \]
for all $i\geq 1$, with $g_1(z) = z$. 

We now want to determine $J$ and the $g_i$ functions. To do this we start from (\ref{equilibre BD-Depoly flux J(z) 2}), divide the 
two members of this equality by $a_iQ_iz^{i+1}$, and using (\ref{defi Q_i}) we obtain
\begin{equation}
    \frac{J(z)}{a_iQ_iz^{i+1}} = \frac{a_izg_i(z)}{a_iQ_iz^{i+1}} - \frac{b_{i+1}g_{i+1}(z)}{a_iQ_iz^{i+1}} 
    = \frac{g_i(z)}{Q_iz^i} - \frac{g_{i+1}(z)}{Q_{i+1}z^{i+1}} 
    \label{equilibrium BD-Depoly eq1 v2}
\end{equation}
Summing from $i=1$ to $k-1$, we obtain
\begin{equation}
    J(z)\sum_{i=1}^{k-1} \frac{1}{a_iQ_iz^{i+1}} = 1 - \frac{g_k(z)}{Q_kz^k}.
    \label{equilibrium BD-Depoly eq2 v2}
\end{equation}
Rewriting (\ref{equilibrium BD-Depoly eq2 v2}), we obtain the expression of $g_k(z)$ for all $k\geq 2$:
\begin{equation}    
    g_k(z) = Q_kz^k\left[ 1 - J(z)\sum_{i=1}^{k-1}\frac{1}{a_iQ_iz^{i+1}} \right],
    \label{equilibrium BD-Depoly definition g_k(z)}
\end{equation}
and we also find $g_1(z) = z$.

In order for the sequence $(g_k(z))_{k\geq 1}$ to be a steady state, it needs to be summable against $(a_k)_{k\geq 1}$. We have
\begin{equation*}
    \sum_{k=2}^{+\infty}a_kg_k(z) = \sum_{k=2}^{+\infty} a_kQ_kz^k - J(z)\sum_{k=2}^{+\infty}a_kQ_kz_k\left(\sum_{i=1}^{k-1}\frac{1}{a_iQ_iz^{i+1}}\right).
\end{equation*}
From the definition of $z_s$ as the radius of convergence of a power series, we observe that if $z>z_s$ then the series $\sum a_kQ_kz^k$ 
is divergent, therefore there is no steady state in the super-critical case. We now suppose that $z\leq z_s$. We have for all $k\geq 1$,
\begin{align}
    a_kg_k(z) = a_kQ_kz^k\left[ 1 - J(z) \sum_{i=1}^{k-1}\frac{1}{a_iQ_iz^{i+1}} \right] &> |J(z)|a_kQ_kz_k\sum_{i=1}^{k-1}\frac{1}{a_iQ_iz^{i+1}} \label{preuve equilibre a_kg_k}\\
    &= |J(z)| \frac{a_k}{b_k} + \sum_{i=1}^{k-2}\frac{a_kQ_k}{a_iQ_i}z^{k-(i+1)} \nonumber
\end{align}
Therefore using \ref{(H2)} and \ref{(H2bis)}, if $z_s<+\infty$ then $a_kg_k(z)$ does not converge towards 0, and therefore the 
series diverge. This implies that there is no steady-state in the super-critical case, and no other steady-state in the sub-critical 
case when $z_s$ is finite. \newline

It remains to prove the uniqueness of the steady-state or the existence of an infinite number of steady-state depending on the 
comparison of the kinetics rates. \newline
\textit{First case:} $b_i \leq i a_i$ for $i \gg 1$. Recalling the expression (\ref{equilibrium BD-Depoly definition g_k(z)}) 
    we show that in this case they are not summable against $(a_k)_{k\geq 1}$. Let $h(x)=\sum_{k=1}^{+\infty} a_kQ_kx^k$ for $x\in [0,z_s)$,
    noticing that all derivative of $h$ are increasing in $x$ and using (\ref{defi Q_i}), we have
    \begin{align*}
        \sum_{k=2}^{+\infty} a_kg_k(z) &\geq |J|\sum_{k=2}^{+\infty}a_kQ_kz^k\sum_{i=1}^{k-1}\frac{1}{a_iQ_iz^{i+1}} 
        = |J|\sum_{i=2}^{+\infty} \frac{1}{a_iQ_iz^{i+1}}\sum_{k=i+1}^{+\infty}a_kQ_kz^k \\
        &= |J|\sum_{i=2}^{+\infty} \frac{1}{a_iQ_iz^{i+1}}\int_0^z h^{(i+1)}(t)\frac{(z-t)^i}{i!}dt 
        \geq \frac{|J|}{z} \sum_{i=2}^{+\infty} \frac{1}{a_iQ_iz^i}\frac{h^{(i+1)}(0)}{i!} \int_0^z (z-t)^i dt \\
        &= \frac{|J|}{z}\sum_{i=2}^{+\infty} \frac{(i+1)a_{i+1}}{b_{i+1}} \int_0^z \left(1-\frac{t}{z}\right)^i dt 
        = |J|\sum_{i=2}^{+\infty} \frac{a_{i+1}}{b_{i+1}} \\
        &\geq |J|\sum_{i=2}^{+\infty} \frac{1}{i} = +\infty.
    \end{align*}
    Which proves that in this case, there is no other steady-state. \newline
\textit{Second case:} $b_i > i^\nu a_i$ for $i\gg 1$ with $\nu > 1$. Using the relation (\ref{equilibre flux J}), we obtain 
    for $J<0$
    \begin{equation*}
        C_{i+1} = \frac{a_izC_i+|J|}{b_{i+1}} ~\forall i\geq 1 \text{ and } C_1 = z.
    \end{equation*}
    Solving the recurring sequence, we obtain for $i\geq 2$
    \begin{equation*}
        C_i = |J|\sum_{k=0}^{i-2} \frac{Q_i}{a_{i-1-k}Q_{i-1-k}}z^k + Q_iz^i.
    \end{equation*}
    First, we prove that this sequence is summable against $(a_i)_{i\geq 1}$. Since $\sum a_iQ_iz^i$ is finite for $z<z_s$, it 
    remains to prove that the first term is summable against $(a_i)_{i\geq 1}$. We have
    \begin{equation*}
        \sum_{i=1}^{+\infty}a_i\sum_{k=0}^{i-2} \left(\frac{Q_i}{a_{i-1-k}Q_{i-1-k}}\right)z^k = \sum_{k=0}^{+\infty} \sum_{i=k+2}^{+\infty} \frac{a_iQ_i}{a_{i-1-k}Q_{i-1-k}}z^k = \sum_{k=0}^{+\infty} \underbrace{\left[\sum_{j=2}^{+\infty}\frac{a_{j+k}Q_{j+k}}{a_{j-1}Q_{j-1}}\right]}_{:= \gamma_k}z^k
    \end{equation*}
    For all $k\geq 0$, we prove that $\gamma_k < +\infty$. Let $k\geq 0$, there exists $j_0 \gg 1$ such that for $j\geq j_0$ we have 
    $Q_{j+k} \leq Q_j$. Thus, using \ref{(H2bis)} the sequence of ratio is bounded by $K \geq 0$ and the hypothesis of the second case, we have
    \begin{equation*}
        \sum_{j\geq j_0} \frac{a_{j+k}Q_{j+k}}{a_{j-1}Q_{j-1}} \leq \sum_{j\geq j_0} \frac{a_{j+k}Q_{j}}{a_{j-1}Q_{j-1}} = \sum_{j\geq j_0} \frac{a_{j+k}}{b_j} \leq \sum_{j\geq j_0} \frac{a_{j+k}}{j^\nu a_j} \leq K^{k+1} \sum_{j\geq j_0} \frac{1}{j^\nu} < +\infty.
    \end{equation*}
    This is where the hypothesis of the second case, mainly $\nu > 1$, is crucial. Otherwise, $\gamma_k = +\infty$ and the computation 
    do not allow us to conclude. Computing the ratio of $\gamma_k$, we have
    \begin{align*}
        \frac{\gamma_{k+1}}{\gamma_k} = \frac{\sum\limits_{j=2}^{+\infty}\dfrac{a_{j+k+1}Q_{j+k+1}}{a_{j-1}Q_{j-1}}}{\sum\limits_{j=2}^{+\infty}\dfrac{a_{j+k}Q_{j+k}}{a_{j-1}Q_{j-1}}} = \frac{\sum\limits_{j=2}^{+\infty}\dfrac{a_{j+k}Q_{j+k}}{a_{j-1}Q_{j-1}}\dfrac{a_{j+k+1}}{b_{j+k+1}}}{\sum\limits_{j=2}^{+\infty}\dfrac{a_{j+k}Q_{j+k}}{a_{j-1}Q_{j-1}}}.
    \end{align*}    
    However, for $k\gg 1$ using the hypothesis of the second case, we have
    \begin{equation*}
        \frac{a_{j+k+1}}{b_{j+k+1}} < \frac{a_{j+k+1}}{(j+k+1)^\nu a_{j+k+1}} < \frac{1}{(k+1)^\nu}.
    \end{equation*}
    Therefore,
    \begin{equation*}
        \frac{\gamma_{k+1}}{\gamma_k} \leq \frac{1}{(k+1)^\nu} \xrightarrow{k\to +\infty} 0.
    \end{equation*}
    Using D'Alembert's ratio test, the sum $\sum a_iC_i$ converges for all $z\geq 0$. To obtain a steady-state it remains to satisfy 
    the first equation and if there is any condition on $\lambda$, $z$ or $J$. The equation
    \begin{equation*}
        a_1 z^2 + z |J| \sum_{k=0}^{+\infty} \gamma_k z^k + \sum_{k=1}^{\infty} a_kQ_kz^{k+1}  = \lambda
    \end{equation*}
    admits, for all $J<0$, a unique $z$ such that the above equation is satisfied, since the left-hand side of the equation is a 
    strictly increasing function of $z$ from 0 to $+\infty$. We therefore obtain an infinity of steady-states parameterized by $J$.

\begin{rmk}
    Note that every sequence, even increasing sequence, cannot be compared in the previous two cases. However, both 
    cases already show the critical point at which the existence of a unique or of multiple steady-states change drastically. 
    In the case of power laws coefficients, the two cases are exhaustive. \newline
\end{rmk}

\subsection{Local exponential stability : sub-critical case}

Let $\Theta_i = Q_iz^i$ for all $i\geq 1$. We introduce the ansatz $C_i = \Theta_i(1+h_i)$ where $h_i$ denotes the fluctuation 
around  $\Theta_i$. We introduce the Hilbert space $\mathcal{H} := \ell^2(\Theta_i)$, and a subspace $\mathcal{H}_\mathcal{D} = \ell^2(\Theta_i(1+\sigma_i^2))$ 
namely
\begin{equation*}
    \mathcal{H} = \left\{ h = (h_i)_{i\geq 1} : \|h\|_{\mathcal{H}} := \left(\sum_{i=1}^{+\infty} \Theta_ih_i^2\right)^{1/2} < \infty \right\},
\end{equation*}
and
\begin{equation*}
    \mathcal{H}_\mathcal{D} = \left\{ h = (h_i)_{i\geq 1} : \|h\|_{\mathcal{H}_\mathcal{D}} := \left(\sum_{i=1}^{+\infty} \Theta_i(1+\sigma_i^2)h_i^2\right)^{1/2} < \infty \right\},
\end{equation*}
where the sequence $(\sigma_i)_{i\geq 1}$ is defined later in (\ref{sigma_i}). The inner scalar product of $\mathcal{H}$ is denoted by 
$\langle \cdot,\cdot \rangle_{\mathcal{H}}$. We define two linear operators $S$ and $P$ on $\mathcal{H}$ by the following weak form
\begin{equation}
    \sum_{i=1}^{+\infty} S_i(h)\Theta_i\varphi_i = \sum_{i=1}^{+\infty} [\varphi_{i+1}-\varphi_i-\varphi_1]a_i\Theta_1\Theta_i(h_i+h_1-h_{i+1}),
    \label{equilibrium sub-critical weak form S}
\end{equation}
and
\begin{equation}
    \sum_{i=1}^{+\infty} P_i(h)\Theta_i\varphi_i = - \varphi_1\left(a_1\Theta_1^2h_2+\sum_{i=1}^{+\infty}a_i\Theta_1\Theta_ih_{i+1}\right),
    \label{equilibrium sub-critical weak form P}
\end{equation}
for any compactly supported sequences $(h_i)_{i\geq 1}$ and $(\varphi_i)_{i\geq 1}$. We denote by $L$ the sum of the previous operators, 
namely
\begin{equation}
    L = S + P.
    \label{equilibrium sub-critical L=S+P}
\end{equation}
We also define a bilinear operator $\Gamma$ by the 
following weak form
\begin{equation*}
    \sum_{i=1}^{+\infty} \Gamma_i(f,g)\Theta_i\varphi_i = \frac{1}{2}\sum_{i=1}^{+\infty}a_i\Theta_1\Theta_i(f_1g_i+f_ig_1)(\varphi_{i+1}-\varphi_i-\varphi_1),
\end{equation*}
for any compactly supported sequences $(f_i)_{i\geq 1}$, $(g_i)_{i\geq 1}$ and $(\varphi_i)_{i\geq 1}$. We note that $\Gamma(0,0)=0$.
\newline

The fluctuation $h$ satisfies 
\begin{equation}
    \frac{d}{dt}h_i(t) = L_i(h(t)) + \Gamma_i(h(t),h(t)).
    \label{equilibrium sub-critical ODE fluctuaction}
\end{equation}
The aforementioned decomposition of the linear operator $L$ is due to the fact that the operator $S$ is the linear operator studied by 
Cañizo and Lods \cite{canizo_exponential_2013} for the Becker-Döring equations. Consequently, it is established that $S$ is a symmetric 
operator in $\mathcal{H}$ and dissipative. Cañizo and Lods proved that the operator $S$ admits a spectral gap in $\mathcal{H}$, and 
they have also proved a lower bound on this spectral gap. In order to prove that the steady-state is then locally exponentially stable, 
it is necessary to establish estimates on the quadratic term $\Gamma$. However, finding estimates on the quadratic term $\Gamma$ in a 
weighted $\ell^2$ space is quite complicated. The approach proposed by Cañizo and Lods \cite{canizo_exponential_2013} is to extend the 
operator $S$ in some weighted $\ell^1$ space, with the aim of showing that the operator S still has a spectral gap in this space. 
Then we control the quadratic term is this weighted $\ell^1$ space, which is more treatable.

Our goal is to prove the same spectral gap holds whenever we control the perturbation operator $P$. \newline

\begin{rmk}
    We may also write the operator $S$ in the strong from, for all compactly supported sequence $h = (h_i)_{i\geq 1}$,
    \begin{equation}
        S_1(h) = -\frac{1}{\Theta_1}\left(W_1(h)+\sum_{k=1}^{+\infty}W_k(h)\right), ~~~ S_i(h) = \frac{1}{\Theta_i}\left(W_{i-1}(h)-W_i(h)\right), ~i\geq 2,
        \label{equilibrium sub-critical strong form S}
    \end{equation}
    where
    \begin{equation}
        W_k(h) = a_k\Theta_1\Theta_k(h_k+h_1-h_{k+1}).
        \label{equilibrium sub-critical def W_k}
    \end{equation}
    As well as the operator $L$, for all compactly supported sequence $h = (h_i)_{i\geq 1}$,
    \begin{equation*}
        L_i(h) = -\sigma_ih_i + \sum_{j=1}^{+\infty}\xi_{i,j}h_j, ~\text{ for } i\geq 1,
    \end{equation*}
    where $\sigma_i$ are defined by
    \begin{equation}
        \sigma_1 = 3a_1z + \sum_{i=1}^{+\infty} a_i\Theta_i, ~~ \sigma_i = a_iz+b_i ~\text{ for }i\geq 2,
        \label{sigma_i}
    \end{equation}
    and $\xi_{i,j} = 0$ for $j\not\in \left\{ 1,i-1,i+1\right\}$. The other $\xi_{i,j}$ are not equal to zero but are not relevant in the 
    following. \newline
\end{rmk}

The operator $L$ is in general not continuous on $\mathcal{H}$, however we give a dense subspace $\mathcal{H}_\mathcal{D}$ of 
$\mathcal{H}$ in which it is bounded:
\begin{lem}
    Assume hypotheses \ref{borne a_i sous lineaire et positivite b_i}-\ref{(H4)}. Then there exists a constant $K>0$, depending only on $z$ and the coefficients $(a_i)_{i\geq 1}$ 
    and $(b_i)_{i\geq 1}$, such that for any compactly supported sequence $h = (h_i)_{i\geq 1}$ we have
    \begin{equation*}
        \|L(h)\|_\mathcal{H} \leq K \|h\|_{\mathcal{H}_\mathcal{D}}.
    \end{equation*}
\end{lem}

\begin{proof}
    Using triangular inequality, it is sufficient to prove this bound separately for $S$ then $P$. We already know from Cañizo and 
    Lods \cite[Lemma 2.4]{canizo_exponential_2013} that this is true for the operator $S$. We prove it for the operator $P$. We show 
    equivalently that 
    \[ \left| \langle P(h), \phi \rangle_\mathcal{H} \right| \leq K \|h\|_{\mathcal{H}_\mathcal{D}}\|\phi\|_{\mathcal{H}}, \text{ for all } \phi \in \mathcal{H}. \]
    We have from the weak form of $P$ (\ref{equilibrium sub-critical weak form P}), and using Cauchy-Schwarz's inequality,
    \begin{align*}
        \left| \langle P(h), \phi \rangle_\mathcal{H} \right| &\leq \sqrt{\Theta_1}\|\phi\|_{\mathcal{H}}\left(a_1\Theta_1|h_2| + \sum_{i=1}^{+\infty}a_i\Theta_i|h_{i+1}| \right) \\
        &\leq K\|\phi\|_{\mathcal{H}}\left(\sum_{i=1}^{+\infty}\Theta_i\right)^{\frac{1}{2}}\left(\sum_{i=1}^{+\infty}a_i^2\Theta_i|h_{i+1}|^2\right)^{\frac{1}{2}}.
    \end{align*}
    However, using \ref{(H2)} and the definition of $(\sigma_i)_{i\geq 1}$ (\ref{sigma_i}), we have
    \begin{align*}
        \left(\sum_{i=1}^{+\infty}a_i^2\Theta_i|h_{i+1}|^2\right)^{\frac{1}{2}} &= \sum_{i=1}^{+\infty} \frac{\Theta_{i+1}}{\Theta_i}b_{i+1}^2\Theta_{i+1}|h_{i+1}|^2 \\
        &\leq K\sum_{i=2}^{+\infty}\sigma_i^2\Theta_i|h_i|^2 \leq K\|h\|_{\mathcal{H}_\mathcal{D}}.
    \end{align*}
    Combining the previous inequalities, we obtain the desired estimation on $P$ and therefore on $L$.
\end{proof}

\begin{prop}
    Assume \ref{borne a_i sous lineaire et positivite b_i}-\ref{(H4)}. If (\ref{equilibrium sub-critical condition P-norm spectral gap}) is satisfied, then the operator $L$ 
    admits a spectral gap $\mu_1>0$.
\end{prop}

\begin{proof}
    As mentioned earlier, the operator $S$ has been studied by Cañizo and Lods \cite{canizo_exponential_2013}. They proved that the 
    eigenvalues are non-positive and that $S$ admits a spectral gap $\mu_0>0$ when $z \in (0,z_s)$. Moreover, they found \cite[Theorem 2.15]{canizo_exponential_2013}
    a lower bound on the spectral gap; namely
    \[ \frac{1}{4D}\leq \mu_0, \]
    where 
    \begin{equation}
        D := \sup_{k\geq 1} \left(\sum_{j=k+1}^{+\infty} Q_jz^j\right)\left(\sum_{j=1}^k\frac{1}{a_jQ_jz^j}\right) \in (0,+\infty).
        \label{equilibrium sub-critical definition D}
    \end{equation}
    The fact that $D$ is finite comes from our hypotheses on the coefficients \ref{borne a_i sous lineaire et positivite b_i}-\ref{(H4)}. 
    Moreover, for all $h\in \mathcal{H}$, using the definition of $P$ and \ref{(H4)}, we have
    \begin{align*}
        \|P(h)\|_{\mathcal{H}}^2 \leq \frac{4b^2}{z}\left(\sum_{k=1}^{+\infty}k^\beta\Theta_kh_k\right)^2.
    \end{align*}
    Then using Cauchy-Schwarz's inequality, we obtain
    \begin{align*}
        \|P(h)\|_{\mathcal{H}}^2 \leq \frac{4b^2}{z}\left(\sum_{k=1}^{+\infty}k^{2\beta}Q_kz^k\right)\|h\|_{\mathcal{H}}^2,
    \end{align*}
    this means that
    \begin{equation}
        \opnorm{P} \leq \frac{2b}{\sqrt{z}}\sqrt{\sum_{k=1}^{+\infty}k^{2\beta}Q_kz^k} < +\infty.
        \label{equilibrium sub-critical upper bound P operator}
    \end{equation}
    We now compare the bound on the norm of $P$ and the one on the spectral gap $\mu_0$. A sufficient condition for the norm $\opnorm{P}$ 
    to be smaller than the spectral gap $\mu_0$ is
    \begin{equation*}
        \frac{1}{4D} > \frac{2b}{\sqrt{z}}\sqrt{\sum_{k=1}^{+\infty}k^{2\beta}Q_kz^k},
    \end{equation*}
    which is exactly
    \begin{equation}
        \frac{8b}{\sqrt{z}}\sqrt{\sum_{i=1}^{+\infty}i^{2\beta}Q_iz^i}\sup_{k\geq 1} \left(\sum_{j=k+1}^{+\infty} Q_jz^j\right)\left(\sum_{j=1}^k\frac{1}{a_jQ_jz^j}\right)< 1.
        \tag{\ref{equilibrium sub-critical condition P-norm spectral gap}}
    \end{equation}
    Using \cite[Theorem III.1.3]{engel_nagel} on bounded perturbation operator, we have that under the condition 
    (\ref{equilibrium sub-critical condition P-norm spectral gap}) the linearized operator $L$ admits a spectral gap $\mu_1>0$.
\end{proof}

To prove that the steady-state $(Q_iz^i)_{i\geq 1}$ of the full system (\ref{systeme equations BD-Depoly}) is locally exponentially 
stable, we now need to find some estimates on the quadratic term $\Gamma(h,h)$. As previously announced, we extend the operator $L$ in 
a weighted $\ell^1$ space. \newline

Let $\eta \in (0,\frac{1}{2}\log(\frac{z_s}{z}))$. We introduce two subspaces of $\ell^1$, $\mathcal{Z} = \ell^1(e^{\eta i}\Theta_i)$ 
and $\mathcal{Z}_\mathcal{D} = \ell^1((1+\sigma_i)e^{\eta i}\Theta_i)$ namely
\begin{equation*}
    \mathcal{Z} = \left\{ h = (h_i)_{i\geq 1} : \|h\|_{\mathcal{Z}} := \sum_{i=1}^{+\infty} e^{\eta i}\Theta_i|h_i|< +\infty \right\},
\end{equation*}
and
\begin{equation*}
    \mathcal{Z}_\mathcal{D} = \left\{ h = (h_i)_{i\geq 1} : \|h\|_{\mathcal{Z}_\mathcal{D}} := \sum_{i=1}^{+\infty} (1+\sigma_i)e^{\eta i}\Theta_i|h_i| < +\infty \right\}.
\end{equation*}
The choice of $\eta$ ensures that the space $\mathcal{Z}$ is larger than the previous Hilbert space $\mathcal{H} = \ell^2(\Theta)$, since
\begin{equation*}
    \|h\|_{\mathcal{Z}} =  \sum_{k=1}^{+\infty}\exp(\eta i)\Theta_i|h_i| \leq \left(\sum_{i=1}^{+\infty} h_i^2\Theta_i \right)^{\frac{1}{2}}\left(\sum_{i=1}^{+\infty} \exp(2\eta i)\Theta_i\right)^{\frac{1}{2}} = \left(\sum_{i=1}^{+\infty} \exp(2\eta i)\Theta_i\right)^{\frac{1}{2}} \|h\|_\mathcal{H},
\end{equation*}
and the last factor in the previous inequality is finite for $\eta < \frac{1}{2}\log(\frac{z_s}{z})$. Indeed, we have
\begin{equation*}
    \sum_{i=1}^{+\infty}Q_iz^i\exp(2\eta i) = \sum_{i=1}^{+\infty} Q_ir^i \text{ where } r = \exp(2\eta +\log(z)),
\end{equation*}
and recalling that the radius of convergence of the power series $\sum iQ_iz^i$ is $z_s$, then the above sum is finite 
whenever $r<z_s$, this is exactly when $\eta < \frac{1}{2}\log(z_s/z)$. \newline

As before, the operator $L$ is in general not continuous on $\mathcal{Z}$, however we give a dense subspace $\mathcal{Z}_\mathcal{D}$ of 
$\mathcal{Z}$ in which it is bounded:
\begin{lem}
    Assume hypotheses \ref{borne a_i sous lineaire et positivite b_i}-\ref{(H4)}. Then there exists a constant $K>0$, depending only on $(a_i)_{i\geq 1}$, $(b_i)_{i\geq 1}$, 
    $z$ and $\eta$, such that for all compactly supported sequence $h=(h_i)_{i\geq 1}$ we have
    \begin{equation*}
        \|L(h)\|_{\mathcal{Z}} \leq K\|h\|_{\mathcal{Z}_\mathcal{D}}.
    \end{equation*}
\end{lem}

\begin{proof}
    Using triangular inequality, it is sufficient to prove this bound separately for $S$ then $P$. We already know form Cañizo and 
    Lods \cite[Lemma 3.3]{canizo_exponential_2013} that this is true for the operator $S$. We prove it for the operator $P$. Recalling 
    the definition of $P$ and using \ref{(H4)}, we have
    \begin{align*}
        \|P(h)\|_{\mathcal{Z}} &\leq \exp(\eta)z\left(\frac{1}{z}\left(b_2\Theta_2|h_2| + \sum_{k=2}^{+\infty}b_k\Theta_k|h_k|\right)\right) \\
        &\leq 2b\exp(\eta)\sum_{k=1}^{+\infty}k^\eta \Theta_k|h_k|.
    \end{align*}
    Then, using that $k^\beta \leq K(\eta,\beta) \exp(\eta k)$ for all $k\geq 1$ and the positivity of $\sigma_k$ (\ref{sigma_i}) 
    for all $k\geq 1$, we have
    \begin{align*}
        \|P(h)\|_{\mathcal{Z}} &\leq K(\eta,\beta) \sum_{k=1}^{+\infty}\exp(\eta k) \Theta_k|h_k| \\
        &\leq K(\eta,\beta) \sum_{k=1}^{+\infty} (1+\sigma_k)\exp(\eta k) \Theta_k|h_k| = C\|h\|_{\mathcal{Z}_\mathcal{D}}.
    \end{align*}
    Combining this estimation with the one on $S$, we obtain the desired estimation.
\end{proof}

Following Cañizo and Lods \cite{canizo_exponential_2013}, we extend the operator $L$ to an unbounded operator on $\mathcal{Z}$ that 
has a positive spectral gap using techniques based upon a suitable decomposition of the linearized operator into a dissipative part 
and a “regularizing” part. To extend the operator, we use \cite[Theorem 3.1]{canizo_exponential_2013} which is a slight improvement 
over some of the consequences of the work of Gualdani, Mischler and Mouhot \cite{gualdani_factorization_2017}.

\begin{thm}[\textbf{Extension of the spectral gap}]
    Assume hypotheses \ref{borne a_i sous lineaire et positivite b_i}-\ref{(H4)}. Let $\eta \in (0,\frac{1}{2}\log(\frac{z_s}{z}))$. If (\ref{equilibrium sub-critical condition P-norm spectral gap}) 
    is satisfied, then the extension of $L$ on $\mathcal{Z}_\mathcal{D}$ that we note $\Lambda$ generates a strongly continuous semigroup
    $(\exp(t\Lambda))_{t\geq 0}$ on $\mathcal{Z}_\mathcal{D}$ and there exists $\mu_\ast \in (0,\mu_1]$, a constant $K>0$ such that
    \begin{equation}
        \|\exp(t\Lambda)x\|_{\mathcal{Z}} \leq K\exp(-\mu_\ast t)\|x\|_{\mathcal{Z}}, ~\forall t\geq 0 \text{ and } x\in \mathcal{Z}.
    \end{equation}
    \label{thm semigroup extension gap}
\end{thm}

\begin{proof}
    The proof follows the one of \cite[Theorem 3.5]{canizo_exponential_2013}, to that end we apply \cite[Theorem 3.1]{canizo_exponential_2013}.
    The main modification is to add the operator $P$ in the regularizing part of the decomposition of $\Lambda$. We mean that an operator 
    $\mathcal{A}$ is regularizing if for all $h\in \mathcal{Z}$, $\mathcal{A}h \in \mathcal{H}$. We just have to show that this is true 
    for our operator $P$. Let $h\in \mathcal{Z}$, we have
    \begin{align*}
        \|P(h)\|_{\mathcal{H}}^2 &\leq \frac{4b^2}{z}\left(\sum_{k=1}^{+\infty}b_kQ_kz^kh_k\right)^2 \\
        &\leq \frac{4b^2}{z}K(\beta,\eta)\|h\|_\mathcal{Z}^2.
    \end{align*}
    Then, the proof is straightforward following the proof of \cite[Theorem 3.5]{canizo_exponential_2013}. We obtain that our operator 
    has a spectral gap $\mu_\ast \in (0,\mu_1)$.
\end{proof}

\subsubsection{Proof of Theorem \ref{thm local expo convergence}}

As stated before, we can obtain an estimation on the quadratic term in a weighted $\ell^1$ space. Indeed, Cañizo and Lods 
\cite[Proposition 3.2]{canizo_exponential_2013} proved an estimation on the quadratic term $\Gamma$, it reads as follows:
\begin{prop}
    Let $\eta \in (0,\frac{1}{2}\log(\frac{z_s}{z}))$. There is a constant $K>0$, depending only on $(a_i)_{i\geq 1}$, $(b_i)_{i\geq 1}$ 
    and $z$, such that
    \begin{equation}
        \|\Gamma(h,h)\|_{\mathcal{Z}} \leq K\|h\|_{\mathcal{Z}}\|h\|_{\mathcal{Z}_\mathcal{D}},~\forall h\in \mathcal{Z}_\mathcal{D}.
    \end{equation}
    \label{prop estimation Gamma}
\end{prop}

Let $\eta \in (0,\min\{\nu,\frac{1}{2}\log\frac{z_s}{z}\})$, where $\nu$ comes from (\ref{UIT bound exponential moment init}). 
The fluctuation $h$ satisfies (\ref{equilibrium sub-critical ODE fluctuaction}) in $\mathcal{Z}$, namely
\[ \frac{d}{dt}h = \Lambda h + \Gamma(h,h).\]
Therefore, if $(V_t)_{t\geq 0}$ denotes the semigroup generates by $\Lambda$, we have
\[ h(t) = V_th(0) + \int_0^tV_{t-s}\Gamma(h(s),h(s))ds. \]
Using \textbf{Proposition \ref{prop estimation Gamma}} and \textbf{Theorem \ref{thm semigroup extension gap}}, we have for some constant $C>0$,
\begin{align*}
    \|h(t)\|_{\mathcal{Z}} &\leq \|V_th(0)\|_{\mathcal{Z}} + \int_0^t\|V_{t-s}\Gamma(h(s),h(s))\|_{\mathcal{Z}} ds \\
    &\leq K\|h(0)\|_{\mathcal{Z}}\exp(-\mu_\ast t) + K\exp(-\mu_\ast t)\int_0^t \exp(\mu_\ast s)\|\Gamma(h(s),h(s))\|_{\mathcal{Z}} ds \\
    &\leq K\|h(0)\|_{\mathcal{Z}}\exp(-\mu_\ast t) + K\exp(-\mu_\ast t)\int_0^t \exp(\mu_\ast s)\|h(s)\|_{\mathcal{Z}}\|h(s)\|_{\mathcal{Z}_{\mathcal{D}}} ds,
\end{align*}
For any $\upsilon \in (0,\nu-\eta)$, using Cauchy-Schwarz's inequality we have
\begin{align*}
    \|h(t)\|_{\mathcal{Z}_\mathcal{D}} &= \sum_{i=1}^{+\infty} |h_i(t)|i\Theta_i\exp(\eta i) \\
    &\leq \left(\sum_{i=1}^{+\infty} |h_i(t)|i\Theta_i\exp(\eta i-\upsilon i)\right)^{\frac{1}{2}}\left(\sum_{i=1}^{+\infty}|h_i(t)|i\Theta_i\exp(\eta i+\upsilon i)\right)^{\frac{1}{2}}.
\end{align*}
However, $i\exp(-\upsilon i) \leq K$, and $i\exp(\eta i +\upsilon i) \leq K\exp(\nu i)$, for all $i\geq 1$. Therefore, we deduce from 
(\ref{UIT bound exponential moment}) that
\[\|h(t)\|_{{\mathcal{Z}}_\mathcal{D}} \leq K \|h(t)\|_{\mathcal{Z}}^{\frac{1}{2}} \left(\sum_{i=1}^{+\infty}|h_i(t)|\Theta_i\exp(\nu i)\right)^{\frac{1}{2}} \leq K \Xi_1^\frac{1}{2}\|h(t)\|_{\mathcal{Z}}^{\frac{1}{2}}, ~\forall t\geq 0.\]
All-in-one, we have
\begin{equation*}
    \|h(t)\|_{\mathcal{Z}} \leq K\|h(0)\|_X\exp(-\mu_\ast t) + K\exp(-\mu_\ast t)\int_0^t \exp(\mu_\ast s)\|h(s)\|_{\mathcal{Z}}^\frac{3}{2} ds.
\end{equation*}
Let $u(t) = \|h(t)\|_{\mathcal{Z}}\exp(\mu_\ast t)$, we have
\begin{align*}
    u(t) &\leq K\|h(0)\|_{\mathcal{Z}}  + K\int_0^t u(s)\|h(s)\|_{\mathcal{Z}}^\frac{1}{2} ds \\
    &\leq K\varepsilon  + K\int_0^t \exp\left(-\frac{\mu_\ast}{2}s\right)u(s)^\frac{3}{2} ds.
\end{align*}
Using Grönwall's lemma, we have for some $K>0$ and $\varepsilon$ small enough,
\begin{equation*}
    u(t) \leq K.
\end{equation*}
This means that 
\begin{equation*}
    \|h(t)\|_{\mathcal{Z}} \leq K\exp(-\mu_\ast t), ~\forall t\geq 0,
\end{equation*}
which is exactly what we want. Indeed, recalling the ansatz $C_i = \Theta_i(1+h_i)$, we have
\begin{equation*}
    \|h(t)\|_{\mathcal{Z}} = \sum_{i=1}^{+\infty} e^{\eta i}Q_iz^i|h_i(t)| = \sum_{i=1}^{+\infty}e^{\eta i}|C_i(t)-Q_iz^i|,
\end{equation*}
and therefore, the previous inequality is exactly (\ref{local expo convergence}).

\section{Numerical simulations} \label{section : numerical simualations}

\subsection{Numerical scheme}

In this section we describe a numerical scheme used for the conservative truncation (\ref{systeme equations BD-Depoly tronc conserv}). 
At the end, we will make a few comments about the non-conservative truncation. \newline

Let $n\geq 2$ an integer. Recalling the PDE (\ref{PDE type focker planck}) and (\ref{PDE flux J}) it remains to add the boundary 
conditions. 
As we are considering the conservative truncation, we have on the right side of our domain the following boundary condition:
\begin{equation}
    J(t,n) = 0.
    \label{PDE BC right}
\end{equation}
Now, on the left side of our domain, which is the production of monomers, we have the following ODE:
\begin{equation}
    \frac{d}{dt}C(t,1) = \lambda - \int_1^n a(x)C(t,1)C(t,x)dx - a(1)C(t,1)^2.
    \label{PDE BC left}
\end{equation}

\subsubsection{Verification of the choice of PDE}

We verify that with the right discretisation and the right mesh, we retrieve our system of equations (\ref{systeme equations BD-Depoly tronc conserv}). 
To do this, we use a uniform mesh with a size step $\Delta x = 1$, giving
\begin{equation*}
    C_i(t) \approx C(t,i), ~i=1,\ldots,n.
\end{equation*}
We discretize (\ref{PDE type focker planck}) using centre differences
\begin{align*}
    \frac{\partial J(t,x)}{\partial x}\bigg|_{x=x_i} &\approx \frac{J(t,x_{i+1/2})-J(t,x_{i-1/2})}{\Delta x},\\
    \frac{\partial}{\partial x}\left(\frac{C(t,x)}{Q(x)C(t,1)^x}\right)\bigg|_{x=x_{i+1/2}} &\approx \frac{1}{\Delta x}\left( \frac{C_{i+1}(t)}{Q_{i+1}C_1^{i+1}(t)} - \frac{C_i(t)}{Q_iC_1^i(t)}\right). \\
\end{align*} 
Therefore, since $\Delta x = 1$ and from (\ref{PDE flux J discretisation delta x = 1}), we obtain
\begin{equation*}
    \frac{\partial C(t,x)}{\partial t}\bigg|_{x=i} = -\frac{\partial J(t,x)}{\partial x}\bigg|_{x=i} \approx \frac{J_{i-1}(t)-J_i(t)}{\Delta x} = \frac{d}{dt}C_i(t). 
\end{equation*}
We just need to check the case $i=1$ which comes from the left boundary condition (\ref{PDE BC left}). Using the rectangle rule to approximate the integral, we have
\begin{equation*}
    \frac{d}{dt}C(t,1) = \lambda - \int_1^n a(x)C(t,1)C(t,x) dx - a(1)C(t,1)^2 \approx \lambda - \sum_{j=1}^{n-1} a_jC_1(t)C_j(t) - a_1C_1(t)^2 = \frac{d}{dt}C_1(t).
\end{equation*}
We have indeed find back our ODE (\ref{systeme equations BD-Depoly tronc conserv}) with a given scheme and a given mesh.

\subsubsection{Reaction-Diffusion (RD) scheme}

We now follow the idea in \cite{goudon_fokker-planck_2020} on a non-uniform mesh. Let $K>0$ and $\mathbb L$ the set of all the nodes in our non-uniform mesh:
\begin{equation*}
\mathbb{L} = \left\{ n_1,\ldots,n_K ~:~ n_j<n_{j+1},~ n_1=1,~n_K=n \right\}.
\end{equation*}
We then define the following size step:
\begin{equation*}
    \Delta x_{j+1/2} = n_{j+1} - n_j \text{ and } \Delta x_j = n_{j+1/2}-n_{j-1/2} \text{ where } n_{j+1/2} = \frac{n_{j+1}+n_j}{2} \text{ with }n_j,n_{j+1}\in\mathbb{L}.
\end{equation*}

Using finite volume scheme, we discretize (\ref{PDE type focker planck}) and (\ref{PDE flux J}) as follows:
\begin{equation}
    C_{n_j}^{k+1}-C_{n_j}^k = \frac{\Delta t}{\Delta x_j}\left(J_{n_{j-1/2}}^{k+1}-J_{n_{j+1/2}}^{k+1}\right),
    \label{PDE RD scheme non unif}
\end{equation}
with
\begin{equation}
    J_{n_{j-1/2}}^{k+1} = -a_{n_{j-1/2}}Q_{n_{j-1/2}}(C_1^k)^{n_{j-1/2}+1/2}\frac{1}{\Delta x_{j-1/2}}\left(\frac{C_{n_j}^{k+1}}{Q_{n_j}(C_1^k)^{n_j}}-\frac{C_{n_{j-1}}^{k+1}}{Q_{n_{j-1}}(C_1^k)^{n_{j-1}}}\right).
    \label{PDE RD scheme J non unif}
\end{equation}
At this point, it remains to choose what is $a_{n_{j-1/2}}$ and $Q_{n_{j-1/2}}$. To remain consistent with the discrete equations 
(i.e. when we take the step size equal to 1) we impose $a_{n_{j-1/2}} = a_{n_{j-1}}$ et $Q_{n_{j-1/2}} = Q_{n_{j-1}}$. 
We have tried to impose something else, and the numerical results were not satisfying.
Let $M_{n_j}^k = Q_{n_j}(C_1^k)^{n_j}$, we have
\begin{equation*}
    Q_{n_{j-1}}(C_1^k)^{n_{j-1}} = M_{n_{j-1}} = \sqrt{M_{n_{j-1}}M_{n_j}}\sqrt{\frac{M_{n_{j-1}}}{M_{n_j}}} 
    = \sqrt{M_{n_{j-1}}M_{n_j}}\sqrt{\prod_{k=n_{j-1}+1}^{n_j}\frac{b_k}{a_{k-1}}}\left(\sqrt{C_1^k}\right)^{-(n_j-n_{j-1})}.
\end{equation*}
Using the above equality, we can rewrite (\ref{PDE RD scheme non unif}), we obtain for all $2\leq j\leq K-1$,
\begin{multline*}
    C_{n_j}^{k+1}-C_{n_j}^k = \frac{\Delta t}{\Delta x_j}\sqrt{M_j^k}\left(\frac{a_{n_j}\sqrt{C_1^k}}{\Delta x_{j+1/2}}\sqrt{\prod_{k=n_{j}+1}^{n_{j+1}}\frac{b_k}{a_{k-1}}}\frac{C_{n_{j+1}}^{k+1}}{\sqrt{M_{n_j}^k}} + \frac{a_{n_{j-1}}\sqrt{C_1^k}}{\Delta x_{j-1/2}} \sqrt{\prod_{k=n_{j-1}+1}^{n_j}\frac{b_k}{a_{k-1}}}  \frac{C_{n_{j-1}}^k}{\sqrt{M_{n_{j-1}}^k}} \right. \\
    \left. - \left( \frac{a_{n_j}}{\Delta x_{j+1/2}}\left(\sqrt{C_1^k}\right)^{n_{j+1}-n_j+1} + \frac{a_{n_{j-1}}}{\Delta x_{j-1/2}} \left[ \prod_{k=n_{j-1}+1}^{n_j}\frac{b_k}{a_{k-1}}\right]\left(\sqrt{C_1^k}\right)^{1-(n_j-n_{j-1})} \right) \frac{C_{n_j}^{k+1}}{\sqrt{M_{n_j}^k}} \right).
\end{multline*}
Let $h_{n_j}^\ast := \dfrac{C_{n_j}^{k+1}}{\sqrt{M_{n_j}^k}}$ and $h_{n_j} = \dfrac{C_{n_j}^k}{\sqrt{M_{n_j}^k}}$, we then have for all $2\leq j \leq K-1$,
\begin{multline*}
    h_{n_j}^{\ast} = h_{n_j}^k + \frac{\Delta t}{\Delta x_j}\left(\frac{a_{n_j}}{\Delta x_{j+1/2}}\sqrt{\prod_{k=n_{j}+1}^{n_{j+1}}\frac{b_k}{a_{k-1}}}\sqrt{C_1^k}h_{n_{j+1}}^\ast + \frac{a_{n_{j-1}}}{\Delta x_{j-1/2}} \sqrt{\prod_{k=n_{j-1}+1}^{n_j}\frac{b_k}{a_{k-1}}} \sqrt{C_1^k} h_{n_{j-1}}^\ast \right. \\
    \left. - \left( \frac{a_{n_j}}{\Delta x_{j+1/2}}\left(\sqrt{C_1^k}\right)^{n_{j+1}-n_j+1} + \frac{a_{n_{j-1}}}{\Delta x_{j-1/2}} \left[ \prod_{k=n_{j-1}+1}^{n_j}\frac{b_k}{a_{k-1}}\right]\left(\sqrt{C_1^k}\right)^{1-(n_j-n_{j-1})}\right)h_{n_j}^\ast \right).
\end{multline*}
As mentioned in the beginning of this section, we present this scheme on the conservative truncation (\ref{systeme equations BD-Depoly tronc conserv}), 
thus we impose the zero flux condition on the right boundary, meaning $J_{n_{K+1/2}} = 0$. This condition may be rewritten as follows:
\begin{equation}
    \frac{C_{n_{K+1}}^{k+1}}{M_{n_{K+1}}^k} = \frac{C_{n_K}^{k+1}}{M_{n_K}^k},~~ \forall k>0.
    \label{PDE RD scheme non unif nul flux}
\end{equation}
Therefore, the scheme becomes
\begin{equation}
    \begin{aligned}
        C_1^{k+1} &= C_1^k + \Delta t \left( \lambda - C_1^{k+1}\left(\sum_{j=1}^{K-1} a_{n_j}\Delta x_{j+1/2}C_{n_j}^k+a_1C_1^k\right)\right) \\
        h_1^\ast &= \frac{1}{\sqrt{C_1^k}} \frac{C_1^k + \lambda \Delta t}{1+\Delta t\left(\sum\limits_{j=1}^{K-1} a_{n_j}\Delta x_{j+1/2}C_{n_j}^k+a_1C_1^k\right)} \\
        h_{n_j}^{\ast} &= h_{n_j}^k + \frac{\Delta t}{\Delta x_j}\left(\frac{a_{n_j}}{\Delta x_{j+1/2}}\sqrt{\prod_{k=n_{j}+1}^{n_{j+1}}\frac{b_k}{a_{k-1}}}\sqrt{C_1^k}h_{n_{j+1}}^\ast + \frac{a_{n_{j-1}}}{\Delta x_{j-1/2}} \sqrt{\prod_{k=n_{j-1}+1}^{n_j}\frac{b_k}{a_{k-1}}} \sqrt{C_1^k} h_{n_{j-1}}^\ast \right. \\
        & \left. - \left( \frac{a_{n_j}}{\Delta x_{j+1/2}}\left(\sqrt{C_1^k}\right)^{n_{j+1}-n_j+1} + \frac{a_{n_{j-1}}}{\Delta x_{j-1/2}} \left[ \prod_{k=n_{j-1}+1}^{n_j}\frac{b_k}{a_{k-1}}\right]\left(\sqrt{C_1^k}\right)^{1-(n_j-n_{j-1})}\right)h_{n_j}^\ast \right) \\
        h_{n_K}^\ast &= h_{n_K} + \frac{\Delta t}{\Delta x_K}\left( \frac{a_{n_{K-1}}}{\Delta x_{K-1/2}}\sqrt{\prod_{k=n_{K-1}+1}^{n_K} \frac{b_k}{a_{k-1}}}\sqrt{C_1^k}h_{n_{K-1}} \right. \\
        & \hspace{4.7cm} \left. - \frac{a_{n_{K-1}}}{\Delta x_{K-1/2}}\left[\prod_{k=n_{K-1}+1}^{n_K} \frac{b_k}{a_{k-1}} \right] \left(\sqrt{C_1^k}\right)^{1-(n_K-n_{K-1})}h_{n_K}  \right).
    \end{aligned}
    \label{PDE RD scheme h non unif full}
\end{equation}

Numerically to implement (\ref{PDE RD scheme h non unif full}), we solve the following linear system:
\begin{equation}
    \left(I_{K-1} - \Delta t\mathsf{W}^k\right)H^\ast = H + \beta^k,
    \label{PDE scheme RD system lineaire non unif}
\end{equation}
where
\begin{equation*}
H^\ast = 
\begin{pNiceMatrix}[nullify-dots,xdots/line-style=loosely dotted]
h_2^\ast \\ \Vdots \\ \Vdots \\ \Vdots \\ h_{n_K}^\ast
\end{pNiceMatrix}, ~~
H = 
\begin{pNiceMatrix}[nullify-dots,xdots/line-style=loosely dotted]
h_2 \\ \Vdots \\ \Vdots \\ \Vdots \\ h_{n_K}
\end{pNiceMatrix}, ~~
\beta^k = \frac{\Delta t}{\Delta x_2 \Delta x_{2-1/2}}a_{n_1}\sqrt{C_1^k}\sqrt{\prod_{k=n_1+1}^{n_2}\frac{b_k}{a_{k-1}}}
\begin{pNiceMatrix}[nullify-dots,xdots/line-style=loosely dotted]
h_1^\ast \\ 0 \\ \Vdots \\ \Vdots \\ \Vdots \\ 0
\end{pNiceMatrix},
\end{equation*}
and the tri-diagonal matrix
\begin{equation*}
\mathsf{W}^k =
\begin{pNiceMatrix}[nullify-dots,xdots/line-style=loosely dotted]
-\omega_{n_2}^k & \alpha_{n_{2+1/2}} & 0 & & \Cdots & 0 & 0\\
\gamma_{n_{3-1/2}} & -\omega_{n_3}^k & \alpha_{n_{3+1/2}}  & \Ddots & & \Vdots & \Vdots \\
0 & \gamma_{n_{4-1/2}}  & -\omega_{n_4}^k & \Ddots & & &\\
& \Ddots & \Ddots & \Ddots & & 0 &\\
\Vdots & & & & & \alpha_{n_{K-3/2}} & 0\\
0 & \Cdots & & 0 & \gamma_{n_{K-3/2}}& -\omega_{n_{K-1}}^k & \alpha_{n_{K-1/2}} \\
0 & \Cdots & &   & 0        & \gamma_{n_{K-1/2}} & -\nu_{n_K}^k
\end{pNiceMatrix},
\end{equation*}
with for $2\leq j\leq K-1$,
\begin{equation*}
    \omega_{n_j}^k = \frac{1}{\Delta x_j}\left(\frac{a_{n_{j}}}{\Delta x_{j+1/2}}\left(\sqrt{C_1^k}\right)^{n_{j+1}-n_j+1} + \frac{a_{n_{j-1}}}{\Delta x_{j-1/2}} \left[ \prod_{k=n_{j-1}+1}^{n_j}\frac{b_k}{a_{k-1}}\right]\left(\sqrt{C_1^k}\right)^{1-(n_j-n_{j-1})} \right),
\end{equation*}
\begin{equation*}
    \alpha_{n_{j+1/2}} = \frac{1}{\Delta x_j}\frac{a_{n_j}}{\Delta x_{j+1/2}}\sqrt{\prod_{k=n_{j}+1}^{n_{j+1}}\frac{b_k}{a_{k-1}}}\sqrt{C_1^k}, ~~  
    \gamma_{n_{j-1/2}} = \frac{1}{\Delta x_{j}}\frac{a_{n_{j-1}}}{\Delta x_{j-1/2}} \sqrt{\prod_{k=n_{j-1}+1}^{n_j}\frac{b_k}{a_{k-1}}} \sqrt{C_1^k},
\end{equation*}
and
\begin{equation*}
    \nu_{n_K}^k = \frac{1}{\Delta x_K}\frac{a_{n_{K-1}}}{\Delta x_{K-1/2}}\left[\prod_{k=n_{K-1}+1}^{n_K} \frac{b_k}{a_{k-1}} \right] \left(\sqrt{C_1^k}\right)^{1-(n_K-n_{K-1})}.
\end{equation*}

\begin{rmk}
    As opposed to uniform mesh, the matrix $\mathsf{W}^k$ is no longer symmetric, which can translate in 
    a loss of computation time. Therefore, to symmetrise this matrix while keeping the tri-diagonal form we can multiply the matrix by
    \begin{equation*}
        \mathsf{\Theta} = \diag\left(\frac{\Delta x_j}{\Delta x_1}\right)_{2\leq j\leq K}.
    \end{equation*}
\end{rmk}

The scheme on $h^\ast$ (\ref{PDE RD scheme h non unif full}) is unconditionally stable in time and size. Indeed, the only restriction 
could be on the time step because of Euler’s explicit-implicit scheme for $C_1$, but this scheme is also unconditionally stable. 
We have
\begin{equation*}
    C_1^{k+1}\left(1+\Delta t\left(a_1C_1^k+\sum_{j=1}^{K-1}a_{n_j}\Delta x_{j+1/2}C_{n_j}^k\right)\right) = C_1^k + \Delta t \lambda
\end{equation*}
thus
\begin{equation*}
    C_1^{k+1} \leq C_1^k + \Delta t \lambda,
\end{equation*}
which leads to, for $0\leq t^k \leq T$,
\begin{equation*}
    C_1^{k} \leq C_1^0 + \lambda t^k \leq C_1^0 + \lambda T.
\end{equation*}

\begin{rmk}
    We could just do an Euler explicit scheme for (\ref{PDE BC left}), but then we would have two conditions 
    on the time step $\Delta t$: stability and to preserve positivity. This would lead to have an adaptative 
    time step $\Delta t^k$. \newline
\end{rmk}

We now have a well-balanced scheme and unconditionally stable in time and size. However, there is a difference with the uniform 
mesh with $\Delta x = 1$ or with the ODEs which is the parameter $z$ of the steady-state. Indeed, equations (\ref{PDE RD scheme h non unif full})
for $j\geq 2$ implies that the steady-state is of the form $Q_{n_j}\eta^{n_j}$, and the parameter $\eta$ is determined with the 
equation on $C_1$. The difference comes from the following two equations 
\begin{equation}
    \lambda - \sum_{i=1}^{N-1} a_iQ_iz^{i+1} - a_1z^2 = 0 \text{ and } \lambda - \sum_{j=1}^{K-1} a_{n_j}\Delta x_{n_{j+1/2}}Q_{n_j}\tilde z^{n_j+1} - a_{n_1}\tilde z^2 = 0,
    \label{numerical equilibria different z}
\end{equation} 
which do not have the same solutions. In fact, the difference of the two solutions is of order $Q_{n_l}\tilde z^{n_l}$ where $l$ is 
defined as follows:
\[ l = \min \left\{ j \in \left\{1,\ldots,K-1\right\} ~:~ \Delta x_{j+1/2} > 1  \right\}. \]

A difference is expected since we reduce the linear system therefore we lost information which translate into a slight difference 
in the steady-state. We note that we can compute the steady-state on a uniform mesh with $\Delta x = 1$ and therefore choose 
the right mesh to control the error. Morally if we choose a mesh with $\Delta x = 1$ for small size and whatever we want for 
bigger size then we approach the right steady-state.

\begin{rmk}[\textbf{Non-conservative truncation}]
    Changing the right boundary condition (\ref{PDE RD scheme non unif nul flux}) 
    to one taking account things that can leave the system, we obtain the same scheme except for the bottom right term in the iteration 
    matrix. However, the scheme is no longer well-balanced since we still use the same form of steady-state. One modification could be to 
    use the steady-state of the non-conservative truncation and doing again the same scheme.
\end{rmk}

\subsection{Numerical results}

Before showing numerical results, we have to mention the fact that the conservative truncation (\ref{systeme equations BD-Depoly tronc conserv}) 
has the same steady-state as the infinite system, meaning $(Q_iz^i)_{i\geq 1}$, however the $z$ differs due to the truncation. In fact, 
the conservative truncation (\ref{systeme equations BD-Depoly tronc conserv}) has a unique steady-state, which is this one, no matter 
$\lambda$. There is no sub-critical or super-critical cases in any truncation, indeed if $z_s$ is finite then for a fixed $n$ and $\lambda$ 
there always exists $z(n)>z_s$ such that $(Q_iz(n)^i)_{1\leq i\leq n}$ is the only steady-state of (\ref{systeme equations BD-Depoly tronc conserv}).

Following the above discussion on $z$ for different meshes, we denote $z_{RK4} = z_{RDU}$ the parameter $z$ obtain with the full 
sum, meaning with $\Delta x = 1$, and we denote $z_{RDNU}$ the parameter $z$ obtain with a non-uniform mesh. Consequently, we denote 
$C^{eq}_{RK4} = C^{eq}_{RDU}$ and $C^{eq}_{RDNU}$ the corresponding steady-state. \newline

In every numerical simulation used to obtain graphs below, we start with a system containing nothing:
\[ C_i(0) = 0, ~\text{for } i = 1,\ldots,n. \]
For the size discretisation we use a non-uniform mesh with $\Delta x = 1$ close to the boundaries and then nodes are spaced evenly 
on a log scale until the step size reaches the maximum size that we impose. \newline

We start with the convergence towards the steady-state of our numerical solution. \textsc{Figure} \ref{Figure cvg eq} shows results for a 
system of size $n = 30 000$ with kinetics coefficients and production rate as follows:
\begin{equation}
    a_i = i^{1/2}, ~b_i = i^{2/3} \text{ and } \lambda = 10.
    \label{numerical parameters SET2}
\end{equation}
These parameters gives $z_s = +\infty$ and therefore $\lambda_s = +\infty$, and \texttt{z $\approx$ 1.119}. Using our coarse-grain 
mesh, we have $|\texttt{z}_{RDU}-\texttt{z}_{RDNU}| \leq \texttt{1.8e-8}$. 

\textsc{Figure} \ref{Figure cvg eq norme inf} shows that the numerical solution of our scheme convergences towards the steady-state 
$(Q_iz_{RDNU}^i)_{1\leq i\leq n}$ at the given precision, and the green curve show what we expected which is that the two equilibria 
differ at the expected order of magnitude. Focusing on the red and blue curve, our scheme preserves the exponential decay towards 
the steady-state. 
\textsc{Figure} \ref{Figure cvg eq erreur relative} shows that the relative error between the numerical solution obtained from 
our scheme and the one obtained with RK4 scheme (Runge-Kutta method) converge towards the relative error between the different equilibria, which is 
expected from \textsc{Figure} \ref{Figure cvg eq erreur relative} and (\ref{numerical equilibria different z}). However, unlike 
the scheme of Duncan and Soheili \cite[Fig. 3]{duncan_approximating_2001}, the relative error does not behave similarly. Indeed, their 
relative error is close to 0 for small time but remains around $10^{-2}$ when the solution reaches the steady-state, whereas
our relative error is bigger for small time but decrease towards the expected error (the sub-sampling error from 
(\ref{numerical equilibria different z})) as the solution reaches the steady-state.

\begin{figure}[ht]
    \centering    
    \begin{subfigure}[c]{0.49\textwidth}
        \includegraphics[width=\textwidth]{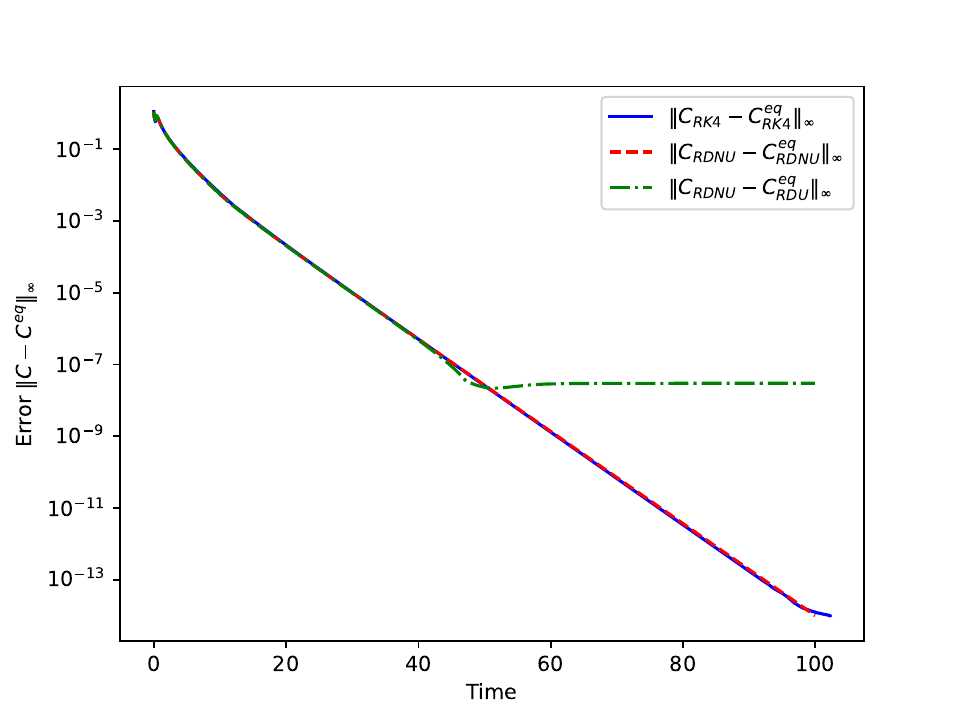}
        \caption{Absolute error}
        \label{Figure cvg eq norme inf}
    \end{subfigure}
    \begin{subfigure}[c]{0.49\textwidth}
        \includegraphics[width=\textwidth]{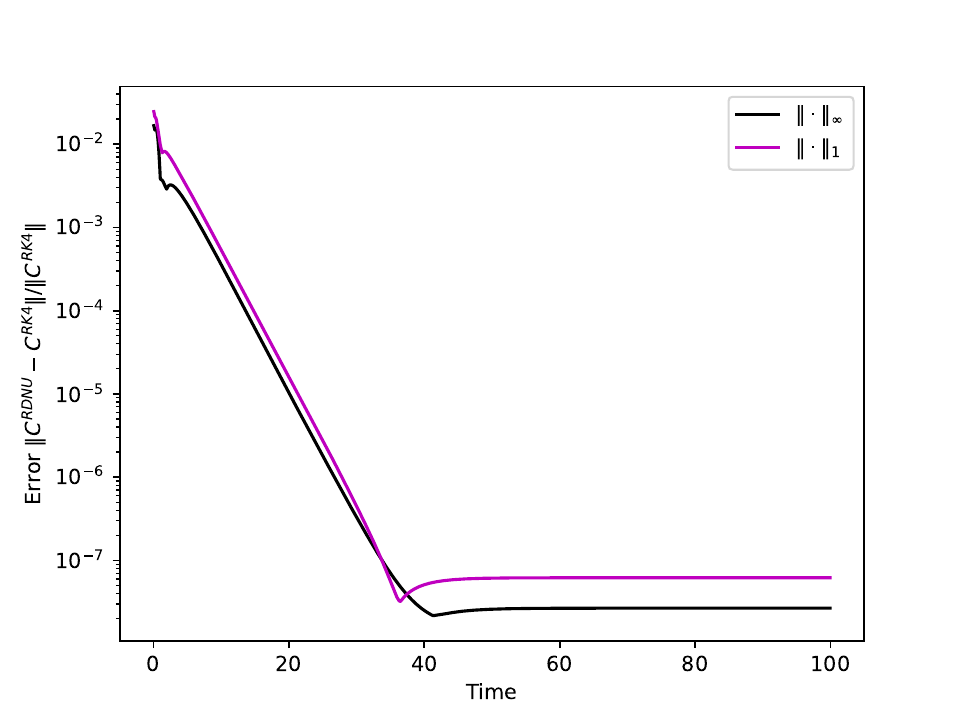}
        \caption{Relative error}
        \label{Figure cvg eq erreur relative}
    \end{subfigure}
    \caption{Comparison of the convergence towards the steady-state. System size $K = 649$ with maximum step size $\Delta x_{max} = 50$.
    The relative error is only taken from the nodes on the non-uniform mesh, no interpolation is used.}
    \label{Figure cvg eq}
\end{figure}

We now compare computation time, since for the super-critical case we may need to compute very large size for a very long time. 
We therefore want a scheme that is able to reduce the computation time. In our scheme, we take $\Delta x_{max}$ such that it 
represents at most 0.05\% of the truncation size $n$. \textsc{Figure} \ref{Figure computation time} shows that, for large enough 
truncation size $n$, it is efficient to use our scheme.

\textsc{Figure} \ref{Figure computation time set2} shows results with kinetics coefficients and production rate from (\ref{numerical parameters SET2}) 
and \textsc{Figure} \ref{Figure computation time set3} shows results with kinetics coefficients and production rate as follows:
\begin{equation}
    a_i = i^{1/2}, ~b_i = 0.1+0.75i^{1/2} \text{ and } \lambda = 1.
    \label{numerical parameters SET3}
\end{equation}
These parameters gives $z_s = 0.75$, \texttt{$\lambda_s$ $\approx$ 5.34375} and \texttt{z $\approx$ 0.516}. 

We mention that with other coefficients (e.g. (\ref{numerical parameters SET4})), the convergence is much slower towards the steady-state and therefore the gain 
can be much better. Taking the maximum size $n=30000$ and allowing only $\Delta x_{max}$ to be roughly $0.05\%$ of $n$, we reduce 
the computational time by $65\%$. Taking $\Delta x_{max}$ to be roughly $0.5\%$ of $n$, the reduction becomes $80\%$. The gain of 
computational time also comes from the time size which is $10$ times bigger for our scheme.
\begin{equation}
    a_i = i^{1/2}, ~b_i = 0.05+0.1i^{2/3} \text{ and } \lambda = 10.
    \label{numerical parameters SET4}
\end{equation}

\begin{figure}[ht]
    \centering
    \begin{subfigure}[c]{0.49\textwidth}
        \includegraphics[width=\textwidth]{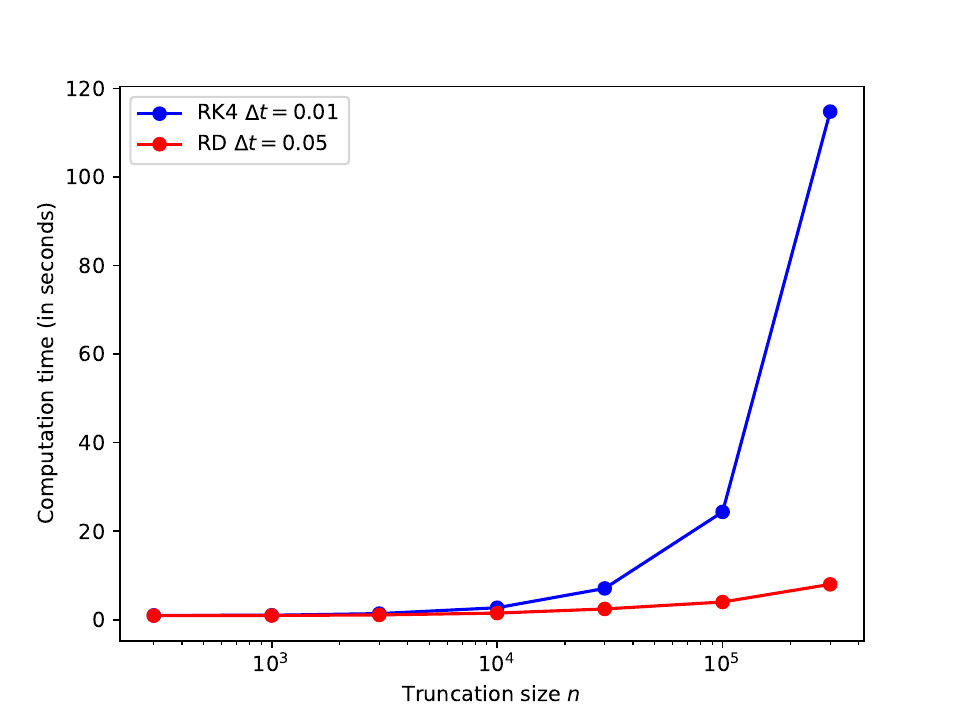}
        \subcaption{Parameters (\ref{numerical parameters SET2})}
        \label{Figure computation time set2}
    \end{subfigure}
    \begin{subfigure}[c]{0.49\textwidth}
        \includegraphics[width=\textwidth]{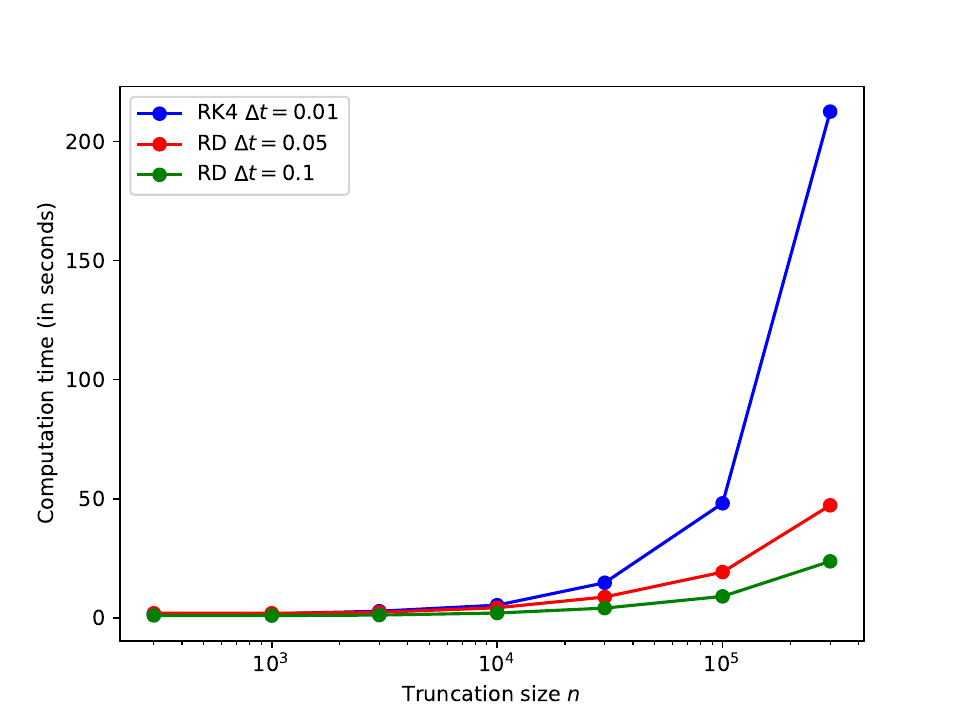}
        \subcaption{Parameters (\ref{numerical parameters SET3})}
        \label{Figure computation time set3}
    \end{subfigure}
    \caption{Comparison of computation time}
    \label{Figure computation time}
\end{figure}

We now compare the dynamics of our numerical solution. We know that the convergence towards the steady-state is the same and that 
for large truncation it is more efficient time wise. However, does the behaviour, quantitatively and qualitatively speaking, of 
our solution is the same or not? \textsc{Figure} \ref{Figure dynamical beaviour} shows that it globally conserves the dynamical 
behaviour of the solution with the same order of magnitude. We only show the first coordinate since the dynamical behaviour is led
by the latter. We note that in those simulations neither $\Delta t$ and $\Delta x_j$ are the same as in the RK4 scheme, which 
already create an error. 

\textsc{Figure} \ref{Figure dynamical beaviour} shows results with kinetics coefficients and production rate from (\ref{numerical parameters SET2}). 
\newline

\begin{figure}[ht]
    \centering
    \begin{subfigure}[c]{0.49\textwidth}
        \includegraphics[width=\textwidth]{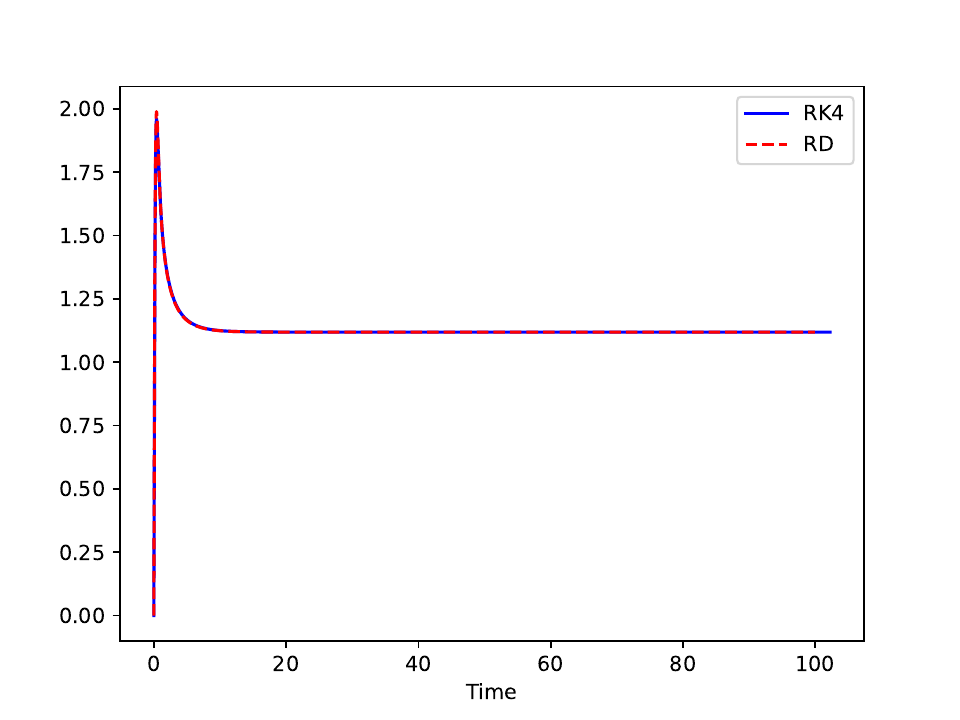}
        \label{Figure dynamical beaviour set2}
    \end{subfigure}
    \begin{subfigure}[c]{0.49\textwidth}
        \includegraphics[width=\textwidth]{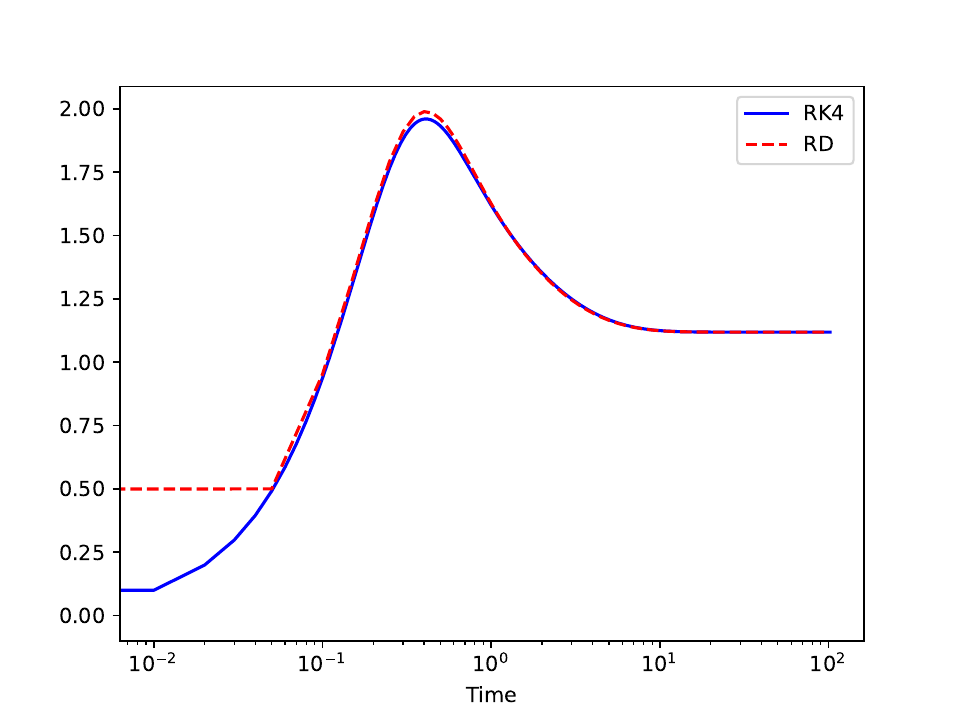}
        \label{Figure dynamical beaviour set2 zoom}
    \end{subfigure}
    \caption{Comparison of the dynamical behaviour of $C_1(t)$}
    \label{Figure dynamical beaviour}
\end{figure}

One more reason we want to develop a different scheme and not settle for classical ODEs schemes is for the super-critical case. 
Indeed, as stated before, either truncation have a unique steady-state, but they are not necessarily the one (or ones) of the 
infinite system and moreover they do not necessarily converge as $n$ goes to infinity to the one (or ones) of the infinite system. 
In practice, we observe that the numerical solution of the RK4 scheme (Runge-Kutta method) always converge towards the steady-state 
of the truncation even in the theoretical super-critical case.

We might expect a different behaviour of our scheme, leading to hints of what is happening in the super-critical case. There is 
no steady-state, but as in the Becker-Döring equations it still could converge towards $(Q_iz_s^i)_{i\geq 1}$ in some sense, 
even if it is not a steady-state unlike in the Becker-Döring equations. We then believe than our scheme may be able to catch a very 
slow convergence towards this sequence that a classical numerical scheme on the truncation system would definitely not catch.

\appendix
\section{Estimations of the sufficient condition}

\subsection{Constant coefficients} \label{annexe constant coeff}

We detail the computation to obtain \textbf{Example \ref{expl constant coeff}}. Recalling that $a_i=a$ and $b_i=b$ for all $i\geq 1$, we have for $z<z_s$,
$\frac{b}{a}=z_s$ and
\begin{align*}
    Q_i = \prod_{j=2}^{i}\frac{a_{j-1}}{b_j}=\prod_{j=2}^{i}\frac{a}{b} = \left(\frac{a}{b}\right)^{i-1} = \frac{1}{z_s^{i-1}}.
\end{align*}
Thus,
\begin{align*}
    \sum_{k=1}^{+\infty} Q_kz^k = z_s\sum_{k=1}^{+\infty} \left(\frac{z}{z_s}\right)^k = \frac{zz_s}{z_s-z},
\end{align*}
and
\begin{align*}
    \sum_{j=1}^k \frac{1}{a_jQ_jz^j} = \frac{1}{a} \sum_{j=1}^k \frac{z_s^{j-1}}{z^j} = \frac{1}{az_s}\frac{\frac{z}{z_s}-\left(\frac{z_s}{z}\right)^{k+1}}{1-\frac{z}{z_s}} = \frac{1}{a}\frac{1-\left(\frac{z_s}{z}\right)^k}{z-z_s},
\end{align*}
and
\begin{align*}
    \sum_{j=k+1}^{+\infty}Q_jz^j = z_s\sum_{j=k+1}^{+\infty}\left(\frac{z}{z_s}\right)^j = \frac{z}{1-\frac{z}{z_s}}-\frac{z-z_s\left(\frac{z}{z_s}\right)^{k+1}}{1-\frac{z}{z_s}} = \frac{z_s^2}{z_s-z}\left(\frac{z}{z_s}\right)^{k+1}.
\end{align*}
We then obtain 
\begin{align*}
    \left(\sum_{j=k+1}^{+\infty}Q_jz^j\right)\left(\sum_{j=1}^k \frac{1}{a_jQ_jz^j}\right) = -\frac{1}{a}\frac{z_s^2}{(z_s-z)^2}\left(\frac{z}{z_s}\right)^{k+1} + \frac{1}{a}\frac{zz_s}{(z_s-z)^2}.
\end{align*}
Therefore, the condition (\ref{equilibrium sub-critical condition P-norm spectral gap}) becomes
\begin{align*}
    \frac{8b}{\sqrt{z}}\sqrt{\frac{zz_s}{z_s-z}} \frac{1}{a}\frac{zz_s}{(z_s-z)^2} < 1,
\end{align*}
which is, after simplification,
\begin{align*}
    \frac{8zz_s^{5/2}}{(z_s-z)^{5/2}}<1.
\end{align*}

\subsection{Linear coefficients} \label{annexe linear coeff}

We detail the computation to obtain the \textbf{Example \ref{expl linear coeff}}. Recalling that $a_i=ai$ and $b_i=bi$ for all $i\geq 1$, we have for $z<z_s$,
$\frac{b}{a}=z_s$ and
\begin{align*}
    Q_i = \prod_{j=2}^{i}\frac{a_{j-1}}{b_j}=\prod_{j=2}^{i}\frac{a(j-1)}{bj} = \left(\frac{a}{b}\right)^{i-1}\prod_{j=2}^{i}\frac{j-1}{j} = \left(\frac{a}{b}\right)^{i-1}\frac{1}{i} = \frac{1}{iz_s^{i-1}}.
\end{align*}
Thus, 
\begin{align*}
    \sum_{i=1}^{+\infty}i^2Q_iz^i &= z\sum_{i=1}^{+\infty}i\left(\frac{z}{z_s}\right)^{i-1} = z\sum_{i=1}^{+\infty}\left[\left(\frac{z}{z_s}\right)^{i}\right]' \\
    &= z\left[\sum_{i=1}^{+\infty}\left(\frac{z}{z_s}\right)^{i}\right]' = z\left(\frac{\frac{z}{z_s}}{1-\frac{z}{z_s}}\right)' \\
    &=\frac{zz_s}{(z_s-z)^2},
\end{align*}
and
\begin{align*}
    \sum_{j=1}^k \frac{1}{a_jQ_jz^j} = \frac{1}{a} \sum_{j=1}^k \frac{z_s^{j-1}}{z^j} = \frac{1}{az_s}\frac{\frac{z}{z_s}-\left(\frac{z_s}{z}\right)^{k+1}}{1-\frac{z}{z_s}} = \frac{1}{a}\frac{1-\left(\frac{z_s}{z}\right)^k}{z-z_s},
\end{align*}
and
\begin{align*}
    \sum_{j=k+1}^{+\infty}Q_jz^j &= \frac{b}{a}\sum_{j=k+1}^{+\infty}\frac{1}{j}\left(\frac{z}{z_s}\right)^j =  z_s\sum_{j=k+1}^{+\infty}\int_0^z \left(\frac{x}{z_s}\right)^{j-1}dx \\
    &= z_s \int_0^z \frac{\left(\frac{x}{z_s}\right)^k}{1-\frac{x}{z_s}}dx. \\
\end{align*}
Therefore, the condition (\ref{equilibrium sub-critical condition P-norm spectral gap}) becomes
\begin{equation*}
    \frac{8z_s^{3/2}}{z_s-z}\sup_{k\geq 1}\left(\int_0^z \frac{\left(\frac{x}{z_s}\right)^k}{1-\frac{x}{z_s}}dx\right)\left(\frac{1-\left(\frac{z_s}{z}\right)^k}{z-z_s}\right) < 1.
\end{equation*}
However, for $z<\dfrac{z_s}{2}$ ($z$ small enough) we have $\dfrac{1}{z_s-z}<\dfrac{2}{z_s}$. Thus, for $k\geq 1$,
\begin{align*}
    \left(\int_0^z \frac{\left(\frac{x}{z_s}\right)^k}{1-\frac{x}{z_s}}dx\right)\left(\frac{1-\left(\frac{z_s}{z}\right)^k}{z-z_s}\right) &< \frac{4}{z_s^2}\int_0^z\left(\frac{x}{z_s}\right)^k\left[\left(\frac{z_s}{z}\right)^k-1\right] \\
    &< \frac{4}{z_s^2}\frac{z_s}{k+1}\left(\frac{z}{z_s}\right)^{k+1}\left[\left(\frac{z_s}{z}\right)^k-1\right] \\
    &< \frac{4}{z_s}\frac{1}{k+1}\frac{z}{z_s} \\
    &< \frac{4z}{z_s^2}.
\end{align*}
Therefore, the condition (\ref{equilibrium sub-critical condition P-norm spectral gap}) is true when
\begin{equation*}
    \frac{32z}{\sqrt{z_s}(z_s-z)}<1.
\end{equation*}

\section*{Acknowledgement}

The author would like to thank Erwan Hingant for many valuable discussions on the topics covered in this paper.

\printbibliography

\end{document}